\documentclass[11pt,reqno]{amsproc}
\usepackage[margin=1in]{geometry}
\usepackage{amsmath, amsthm, amssymb}
\usepackage{hyperref}
\hypersetup{colorlinks=true, pdfstartview=FitV, linkcolor=blue,citecolor=blue, urlcolor=blue}
\usepackage[abbrev,lite,nobysame]{amsrefs}
\usepackage{times}
\usepackage[usenames,dvipsnames]{color}
\usepackage{bbm}
\usepackage{bm}
\usepackage{subcaption}
\usepackage{mathtools}
\usepackage{pdfpages}
\usepackage{comment}
\definecolor{labelkey}{rgb}{0,0,1}

\newcommand{\R}{\mathbb{R}}

\newcommand{\N}{\mathbb{N}}

\newcommand{\C}{\mathbb{C}}

\newcommand{\cC}{\mathcal{C}}

\newcommand{\D}{\mathcal{D}}
\newcommand{\hU}{\mathcal{U}}

\newcommand{\bu}{\bm{u}}

\newcommand{\dd}{\mathrm{d}}

\newcommand{\bv}{\bm{v}}
\newcommand{\bb}{\mathbf{b}}
\newcommand{\de}{\partial}

\newcommand{\wt}[1]{\widetilde{#1}}

\providecommand{\R}{\mathbb{R}}
\providecommand{\C}{\mathbb{C}}
\providecommand{\N}{\mathbb{N}}

\providecommand{\eps}{\varepsilon}

\renewcommand{\leq}{\leqslant}

\providecommand{\Div}{\operatorname{div}}

\newcommand{\mdiv}{\operatorname{\bf{div}}}

\newcommand{\Id}{\operatorname{Id}}

\DeclareMathOperator{\w}{op^w_\eps}

\DeclareMathOperator{\diag}{diag}

\DeclareMathOperator{\tr}{tr}

  \newtheorem{thm}{Theorem}[section]
  
  \newtheorem{Lemma}[thm]{Lemma}

    \newtheorem{Definition}[thm]{Definition}
  \newtheorem{remark}[thm]{Remark}

  \newtheorem{assumption}{Assumption}[section]
   \newtheorem{Question}[thm]{Question}
      \newtheorem{Example}[thm]{Example}

\usepackage{etoolbox}
\patchcmd{\subsubsection}{\itshape}{\itshape\bfseries}{}{} % THIS CHANGE ITALIC FONT OF SUBSUBSECTION INTO BOLD-ITALIC

\title[]{A new look at the controllability cost of linear 
evolution systems with a long gaze at localized data }
\author[R.\ Bianchini]{Roberta Bianchini}
\address{Consiglio Nazionale delle Ricerche, Istituto per le Applicazioni del Calcolo, 00185 Rome, Italy}
\email{roberta.bianchini@cnr.it}
\author[V.\ Laheurte]{Vincent Laheurte}
\address{IMB, Universit\'e de Bordeaux, France}
\email{vincent.laheurte@math.u-bordeaux.fr}
\author[F.\ Sueur]{Franck Sueur}
\address{IMB, Universit\'e de Bordeaux, France}
\email{franck.sueur@math.u-bordeaux.fr}

\begin{document}
\maketitle
%%%%%%%%%%%%%%%%%%%%%%%%%%%%%%%%%%%%%%
\begin{abstract}
We revisit the classical issue of the controllability/observability cost of linear first order evolution systems, starting with ODEs, before turning to some linear first order evolution PDEs in several space dimensions, including hyperbolic systems and pseudo-differential systems obtained by linearization in fluid mechanics. In particular we investigate the cost of localized initial data, and in the dispersive case, of initial data which are semi-classically microlocalized. 
\end{abstract}
%%%%%%%%%%%%%%%%%%%%%%%%%%%%%%%%%%%%%%

\setcounter{tocdepth}{3}
\tableofcontents

\newpage
%%%%%%%%%%%%%%%%%%%%%%%%%%%%%%%%%%%%%%

\section{Reminder on the cost of controllability of linear ODEs}
\label{sec-o}

In the setting of  linear ODEs in finite dimensions, with variable coefficients, the issue of exact controllability can be set as follows. 

\begin{Question}\label{Q1}
Given some time-dependent fields of matrices $A(t)$ and $B(t)$, a positive time $T$, some initial and final states $b_0$ and $b_T$, 
is there a time-dependent field $f(t)$ and a corresponding time-dependent field $b(t)$ such that 
\begin{equation}\label{CODE}
b'(t) +  A(t) b(t) = B(t) f(t)    \text{  for } t \in [0,T]\text{  and } ( b(0) , b(T))=(b_0, b_T) ?
\end{equation}
\end{Question}
 \par 

Above, it is understood that the dependence on the time $t$ is smooth enough for the system to make sense and that the matrices and vectors are compatible in terms of their size. 
A positive answer to  Question \ref{Q1}  means by definition that the differential system $b'(t) =  A(t) b(t)$ is \textit{exactly controllable} in time $T$.
\\  \par 

Let us observe that such a result is only non-trivial whenever the system is under-actuated, that is in the case where the  matrices $B(t)$ are not invertible. Otherwise, it is sufficient to consider any smooth field 
$b(t) $ satisfying $( b(0) , b(T))=(b_0, b_T) $, and then to define the control force field $ f(t) $ by 
\begin{equation*}
 f(t)   :=  B(t)^{ -1 }  \Big( b'(t) +  A(t) b(t) \Big)   \text{  for } t \in [0,T],
\end{equation*}
and then, the answer to   Question \ref{Q1}  is positive. 
\\  \par 

 In the case of a positive answer to  Question \ref{Q1}, one may investigate the minimal quadratic  cost of such a control result.

\begin{Question}\label{Q2}
What is the infimum $\mathcal C=\mathcal C(T,A,B,b_0,b_T)$ of 
$$ \int_0^T |f(t)  |^2 \, dt ,$$
among the control force fields $ f(t) $ which verify  Question \ref{Q1} ?
\end{Question}

The Hilbert Uniqueness Method (HUM), introduced by Jacques-Louis Lions, relates the notions of controllability and observability by means of a duality argument, including the computation of the cost, see for example \cite[Chapter 1]{coron}. 
In the context above, in the particular case where $b(T) = 0$, the so-called case of \textit{null-controllability}, 
this method yields that $\mathcal C$ is the minimal constant for which for any $b_0$, up to renaming the adjoints $A^*(t)$ and $B^* (t)$ respectively by $A(t)$ and $B(t)$, 
the unique solution to the uncontrolled equation: 
\begin{equation}\label{ODE}
b'(t) +  A(t) b(t)   = 0  \text{  for } t \in [0,T]  \text{  and }  b(0) = b_0 ,
\end{equation}
satisfies
\begin{equation}\label{ODE-obs}
|b(T)  |^2  \leq \mathcal C \int_0^T |B(t)  b(t)  |^2 \, dt .
\end{equation}
The interpretation of the inequality above is that the energy $|b(T)  |^2 $ of the final state reached by the uncontrolled system at time $T$ can be bounded by the energy 
$ \int_0^T |B(t) b(t)  |^2 \, dt $ which is observed through the field of matrices  $B(t)$ during the interval $(0,T)$, up to a constant $\mathcal C$. The minimal possible constant $\mathcal C$ in \eqref{ODE-obs} is  also called  the \textit{observability cost} of the differential system: $b'(t) + A(t) b(t) = 0$ on $(0,T)$. 
\\  \par 

Before giving a characterization of $\mathcal C$, 
 let us recall the following elementary result on differential Lyapunov equations.
\begin{Lemma}
\label{lem-lyap}
Let $A$ and $B$ some 
continuous time-dependent fields of matrices.
Then there is a unique global solution $G$ of 
 the differential Lyapunov equation: 
 $$Lyap(A,B) : \quad G'(t) -  G(t) A (t)  - A^* (t) G(t)  =   B^* (t) B (t)    \text{  for } t \in [0,T]  \text{  and }  G(0) = 0.$$
 Moreover at any time $t$, the matrix  $G(t)$ is symmetric and nonnegative.
Finally, for any time $T$, 
\begin{equation}
  \label{rep-gram}
G(T) = \int_0^T R(t,T)^* B^*(t) B(t) R(t,T) \, \dd t,
\end{equation}
 where, for any time $T$, the matrix-valued function 
 $R(\cdot,T)$ is the solution to the linear system
   $$\frac{\dd}{\dd t}R(t,T)=-A(t) R(t,T), \quad R(T,T)=\Id.$$
\end{Lemma}
\begin{proof}
The existence and uniqueness of $G$ and of $R(\cdot,T)$ for any time $T$, as solutions of their respective problems, 
follows from the linear finite-dimensional version of the Cauchy-Lipschitz  theorem.
Then the symmetry follows from the  uniqueness by observing that $G$ and $G^*$ satisfy the same equations. 
Similarly, the formula \eqref{rep-gram} follows from the  uniqueness by observing that the r.h.s. satisfies  the differential Lyapunov equation. 
Finally, the nonnegativity at any time $t\in[0,T]$ of the matrix  $G(t)$ follows from this representation. 
\end{proof}

Let us now recall a characterization of 
 $\mathcal C$ as the inverse of  the minimal eigenvalue $\lambda_\text{min} ({G}(T))$ of 
the so-called  \textit{observability Gramian matrix} $G(T)$. 

\begin{thm}\label{thm0}
The observability cost $\mathcal C$ satisfies $ \mathcal C = \lambda_\text{min} ({G}(T))^{ -1}$, where $G(T)$ is the value at time $T$ of the solution $G$ to the 
differential Lyapunov equation $Lyap(A,B) $.
\end{thm}
A byproduct of Theorem \ref{thm0}  is that a necessary and sufficient condition for a linear time-varying finite-dimensional control system to be observable is that the observability Gramian $G(T)$ is invertible. 
\\  \par 

Let us also mention that for time-invariant finite-dimensional control systems, this condition reduces to the famous Kalman rank condition, an algebraic condition which does not require any integration of linear differential equations.
\\  \par 

Theorem \ref{thm0}  is well-known, see for example \cite[Chapter 1]{coron}. However, we give below a proof which sets up a method which will be repeatedly used in some more difficult settings in the sequel. 

\begin{proof}[Proof of Theorem \ref{thm0}]
The starting idea is that there is an extra integral over time in   \eqref{ODE-obs} so that $ \mathcal C^{ -1}$ can be interpreted as a time. 
To restore this time in the left hand side of  \eqref{ODE-obs}, 
 we introduce the following scalar field: 
\begin{equation}\label{ODE-rec-energy}
E_R (t) :=  G(t) b(t) \cdot   b(t) , 
\end{equation}
where the dot in \eqref{ODE-rec-energy} is the usual euclidean inner product.
The quantity $E_R$ will be called  the \textit{recorded energy}. 
\\  \par 

 In particular the recorded energy starts with zero initial value, that is %
\begin{equation}\label{zero} E_R (0) = 0.
 \end{equation}
  Moreover, by Leibniz' rule and by using the equations of $G$ and $b$,  
its time derivative is given by 
 \begin{align*}
E_R'  
= 
B^*  B    b \cdot   b = |B  b  |^2 ,
\end{align*}
and then,
 \begin{align}\label{crucial}
E_R(T)  =  \int_0^T |B(t)  b(t)  |^2 \, dt .
\end{align}
Thus, at the final time $T$, the recorded energy is equal to the observed energy. 
\\  \par 

On the other hand, from the definition \ref{ODE-rec-energy}, one may compare $E_R(T) $ with the final energy $|b(T)  |^2 $ by
 \begin{align*}
\inf_{b_0 \neq 0}  \frac{ G(T) b(T) \cdot   b(T) }{ |b(T)  |^2 } 
= 
\inf_{b(T) \neq 0}  \frac{ G(T) b(T) \cdot   b(T) }{ |b(T)  |^2 } 
=
 \lambda_\text{min} (G(T)) .
\end{align*}
This concludes the proof of Theorem \ref{thm0}. 
\end{proof}
While the proof above recalls the multiplier method, see for instance \cite{komornik94}, it is interesting to reinterpret in a slightly different way these computations,  more precisely the identity  \eqref{crucial},  for the sequel. In this direction, first, 
 from \eqref{ODE}, we deduce that the rank-one matrix 
$b(t)  \otimes b(t) = b(t)  b(t)^*$, which has to be interpreted as the energy tensor, satisfies the following ODE: 
\begin{equation}\label{MODE}
(b\otimes b)' + A (b\otimes b) +   (b\otimes b) A^* =0 .
\end{equation}
Moreover, the recorded energy reads
\begin{equation}\label{ODE-rec-energy-ma}
E_R  =  G : (b \otimes  b) , 
\end{equation}
where the notation $:$ stands for the Frobenius product, that is for two squares matrices $A$ and $B$ of the same size, 
$$A:B = \tr (A^* B) = B:A.$$
Recalling the following interplay between the Frobenius product and the matrix product:
\begin{equation} \label{rules}
 A : (BC) 
=
 (A C^*) : B
\quad   \text{ and }   \quad 
 A : (CB)
=
 (C^*A ) : B, 
 \end{equation}
we observe that the differential Lyapunov equation $Lyap(A,B)$ 
 is the dual equation of \eqref{MODE} with source term $ B^* B $, meaning that if one tests the differential Lyapunov  equation $Lyap(A,B)$ 
 with a time-dependent  matrix $C(t)$ by means of the Frobenius product and if one integrates by parts in time, then 
\begin{equation}
 \int_0^T G   : (C' +   AC + CA^* ) = (G : C)(T) - \int_0^T B^*  B   : C .
\end{equation}
This way, the identity  \eqref{crucial} follows from choosing $C= b\otimes b$, recalling \eqref{zero}, \eqref{MODE} and \eqref{ODE-rec-energy-ma}.

%%%%%%%%%%%%%%%%%%%%%%%%%%%%%%%%%%%%%%%%%%%%%
\section{Case of a transport PDE}\label{sec:transport}

Let us now turn to the case of a simple PDE: the following linear transport equation:
\begin{equation}\label{toy model}
\de_t\bu + (v\cdot\nabla)\bu + A \bu =  0,  \text{  for } t \in [0,T],
\end{equation}
where the space variable is $x\in \R^d$ with  $d \in \mathbb{N}^*$, 
$v:\R^d \rightarrow \R^d$ is a given regular vector field and 
 $A$ is a given regular field of ${N\times N}$ matrices depending  on $x \in \R^d$, with $N\in\N^*$. These matrices encode some damping or amplification along the transportation by $v$.
The final time $T>0$ is also given and 
 the unknown is the  vector field $\bu:[0,T] \times \R^d\to\R^N$. 
 In this case, the cost is defined as follows.
%%%%%
\begin{Definition}\label{def-cocost}
The \textit{cost of observability} by a regular field of  $n\times N$ matrices $B$  depending  on $x \in \R^d$ 
 is the minimal constant $ \mathcal C=\mathcal C(T,v,A,B)$ such that for any initial condition $\bu_0\in L^2(\R^d)$, the associated solution $\bu(t)$ to  \eqref{toy model} 
 verifies the observability inequality:
 \begin{equation}\label{obs ineq}
 \|\bu (T,\cdot)\|_{L^2(\R^d; \R^N)}^2 \leq \mathcal{C} \int_0^T \|B\bu(t)\|_{L^2(\R^d; \R^n)}^2 \dd t .
\end{equation}
\end{Definition}
%
% \\  \par 
%
  Let us highlight that, in this setting, 
  the HUM method, which was mentioned in the previous section, still holds and the identification of  $ \mathcal C$ with  the cost of controllability of the adjoint system as well. 

The natural question that arises is whether one can give a condition for the observability of system \eqref{toy model} and compute its cost, as done in Theorem \ref{thm0} for ODE systems.
  \\ \par
  The following result, generalizing Theorem \ref{thm0}, shows that $\mathcal C$ can be computed in terms of the infimum over $x_0 \in \R^d$
  of the minimal eigenvalues $ \lambda_\text{min} (G_{x_0}(T))$ of a family of Gramian matrices $G_{x_0}(T)$ 
associated with  time-dependent Lyapunov ODEs involving the coefficients of the PDE  \eqref{toy model} evaluated  along the characteristic flow induced by the vector field $v$. 
This flow is denoted by $X_{x_0} (t)$ and satisfies, for  any ${x_0} \in \R^d$, 
 \begin{equation}\label{flow}
 X'_{x_0} (t) = v(X_{x_0} (t))   \text{  for } t \in [0,T]\text{  and  }  X_{x_0} (0) = {x_0} .
\end{equation}
On the other hand, the Gramian matrices $G_{x_0}$ are given as solutions to $Lyap(A_{{x_0}},B_{x_0})$ 
where, for any ${x_0} \in \R^d$ and for any $t \in [0,T]$,
 \begin{align} \label{coeff-x}
A_{{x_0}} (t)  :=   A(X_{x_0} (t))  -   \frac12  (\Div v)(X_{x_0} (t))  \Id_N   \quad  \text{  and  }    \quad 
B_{x_0} (t)  := B(X_{x_0} (t)) .
\end{align}
\begin{thm}\label{thm00}
The observability cost $\mathcal C$ satisfies 
\begin{equation}\label{eq:cost-sec2}
\mathcal C = \big(\inf_{{x_0} \in \R^d} \, \lambda_\text{min} (G_{x_0}(T)\big)^{ -1}.
\end{equation}
\end{thm}

 A byproduct of Theorem \ref{thm00}  is that a necessary and sufficient condition for  \eqref{toy model}   to be observable by the  field of  matrices $B$ is that the  Gramian matrices $(G_{x_0}(T))_{{x_0} \in \R^d}$ are uniformly invertible.

To prove Theorem \ref{thm00} we will mimick the proof of Theorem \ref{thm0} by 
introducing a recorded energy to which we will apply the following lemma, where we assume that $M$ and $F$ are regular enough.
\begin{Lemma}
\label{lem-duali1}
Let $\bu(t)$ be a solution of  \eqref{toy model}  associated with the initial condition $\bu_0\in L^2(\R^d)$. 
Let  $M$  satisfying the following PDE
\begin{equation} \label{toy multi}
\de_t M+(v\cdot\nabla)M + (\Div v)M-A^*M-MA= F ,
\end{equation}
where   $F(t,x)$ is  a time-dependent family of matrices.
    Let 
    \begin{align}\label{def-em}
    E[M,u] :=  \int_{\R^d}  (\bu \otimes \bu): M   .
\end{align}
Then its time derivative satisfies  for all $t\in [0,T]$,
\begin{align*}
    E[M,u]' =    \int_{\R^d}  (\bu \otimes \bu): F   .
    \end{align*}
\end{Lemma}

\begin{proof}
By Leibniz' rule,  \eqref{toy model} 
twice, integration by parts  and a transposition, 
we obtain that the  time derivative $E[M,u]'$ of $E[M,u]$ 
is given by 
 \begin{align} \label{lab}
E[M,u]'  &=   \int_{\R^d} \Big(
\de_t M+(v\cdot\nabla)M + (\Div v)M-A^*M-MA \Big) \bu \cdot \bu \, dx   .
\end{align}
Then it is sufficient to use \eqref{toy multi} to conclude.
\end{proof}
Another elementary ingredient in the proof of Theorem \ref{thm00} is the following result. 
 \begin{Lemma}\label{min int}
 Let $m:\R^d\to \R^{N\times N}$ be a continuous field of symmetric non negative matrices. Then $$\inf_{\|u\|_{L^2}=1}\int_{\R^d}m(x)u(x)\cdot u(x)\,\dd x=\inf_{x\in\R^d}\lambda_{\rm{min}}(m(x)),$$
 where the infimum on the left hand side is taken over all the vector fields $u$ in  $L^2(\R^d)$
 with $\|u\|_{L^2}=1$.
 \end{Lemma}
\begin{proof}
Let 
$$C:=\inf_{x\in\R^d}\lambda_{\rm{min}}(m(x)).$$
On the one hand, 
for any $u \in L^2(\R^d)$, 
for any $x\in\R^d$, 
  \begin{equation*}
  \label{crmin} m(x)u(x)\cdot u(x) \ge  C |u(x)|^2,\end{equation*}
  so that, by integration, 
\begin{equation}\label{c min}\int_{\R^d}m(x)u(x)\cdot u(x)\,\dd x\ge C \|u\|_{L^2}.\end{equation}
On the other hand, for any $\eps>0$, we choose $x_\eps\in\R^d$ such that $\lambda_{\rm{min}}(m(x_\eps))\le C+\eps$, as well as $b_\eps\in \mathcal S^{N-1}$ an eigenvector of $m(x_\eps)$ associated with $\lambda_{\rm{min}}(m(x_\eps))$, so that 
 \begin{equation*}
 m(x_\eps) b_\eps = \lambda_{\rm{min}}(m(x_\eps)) b_\eps  .
 \end{equation*}
Finally, we set, for any $x\in\R^d$, 
$$u_\eps(x):=(\pi\eps)^{-\frac d4} b_\eps e^{-\frac{|x-x_\eps|^2}{2\eps}}.$$Then, for any $\eps>0$, $\|u_\eps\|_{L^2}=1$ and, since $m$ is continuous, 
$$\lim_{\eps\to0}\int_{\R^d}m(x)u_\eps(x)\cdot u_\eps(x)\,\dd x=C.$$
Combining this equality with \eqref{c min}, we conclude the proof of the lemma.
\end{proof}

We are now equipped to start the proof of Theorem \ref{thm00}.
\begin{proof}[Proof of Theorem \ref{thm00}]
Let $M$ be given by  setting for any $t$ and $x$, 
\begin{equation}\label{compo}
 G_x(t) =  M(t,X_x (t)).
\end{equation}
We highlight  that, for any $t$,  the mapping $x \mapsto X_x (t)$ is invertible, since its Jacobian  satisfies a linear first order  ODE with $1$ as initial condition. 
Then we observe that it follows from the fact that  the Gramian matrices $G_{x}$ are  solutions to $Lyap(A_{{x}},B_{x})$ and from the chain rule that 
$M(t,x)$ is a nonnegative symmetric matrix for any $t$ and $x$, that  $M$ satisfies 
an equation 
of the type \eqref{toy multi}  with $F = B^*B $ and that $M(0,\cdot)=0$. 
Therefore, by the  lemma \ref{lem-duali1} and an integration on $[0,T]$, 
the recorded energy defined by
$ E_R  := E[M,u]$
 satisfies 
 \begin{align}\label{toy estim}
E_R(T) 
  &= \int_0^T  \int_{\R^d} |B\bu(t,x)|^2 \, \dd x  \dd t.
\end{align}
%
% \\  \par 

On the other hand, 
 since the map $\bu_0\mapsto \bu(T,\cdot)$ is a bijection from $L^2(\R^d)$ to itself, 
 by Lemma \ref{min int}, we obtain 
 \begin{equation}
     \label{agaver}
 \inf_{\bu_0\ne0}\frac{ E_R(T)}
 %\int_{\R^d} M(T,x)\bu(T,x)\cdot \bu(T,x)}
 {\|\bu(T,\cdot)\|_{L^2}^2}=\inf_{x\in\R^d}\lambda_{\rm{min}}(M(T,x))
 =\inf_{x\in\R^d}\lambda_{\rm{min}}(G_x(T)),
 \end{equation}
  since  $x\mapsto X_x(T)$ is a bijection from $\R^d$ to itself.
 We conclude the proof of Theorem \ref{thm00} by gathering 
  \eqref{toy estim}
 and
 \eqref{agaver}.
\end{proof}
%%%%%

%%%%%%%%%%%%%%%%%%%%%%%%%%%%%%
\section{A localization in space and amplitude of the  cost}
\label{sec-local}

While the definition of observability cost is rather \emph{global} in space,  we consider below the observability cost 
 \begin{equation}
  \label{costu}
 \mathcal C[\bu_0] := 
  \frac{     \|\bu(T,\cdot)\|_{L^2}^2 }{   \int_0^T\|B\bu(t)\|_{L^2}^2\,\dd t  } ,
  \end{equation}
 of the  initial data $\bu_0$ (at time $T$) 
whose  local energy tensor  $\bu_0 \otimes \bu_0$  are \emph{localized}, for the weak-* topology of the space  $\mathcal M (\R^d ; \R^{N \times N})$ of the matrix-valued measures, at $(b_0 \otimes b_0) \delta_{x_0}$, for some 
 $x_0 \in \R^d$ and $b_0 \in \R^N$. 
Above $\bu$ is the  solution 
 of the transport PDE \eqref{toy model} associated with $\bu_0$.
Since it is a linear equation, the cost $\mathcal C[\bu_0]$ is homogeneous of degree $0$ with respect to the   initial data $\bu_0$. 

%%%%%
\begin{thm}\label{thm00loc}
Let $x_0 \in \R^d$ and set  $G_T := G_{x_0} (T) $, where $ G_{x_0}$ is the solution to the differential Lyapounov equation $Lyap(A_{x_0},B_{x_0})$, where we recall the notation \eqref{coeff-x}.
Let  $b_0 \in \R^N \setminus \{0\}$ and  denote by  
$$ b_T^\mathcal{N} :=  \frac{ b_{x_0,b_0}(T) }{ |    b_{x_0,b_0}(T) |} ,$$
where 
 $b_{x_0,b_0}\in\R^N \setminus \{0\}$ is the solution 
 of 
 the linear ordinary differential equation: 
\begin{equation}
  \label{bxeq}
 b_{x_0,b_0}'  = -A_{x_0}  b_{x_0,b_0}  ,  \quad  \text{ with }  \quad b_{x_0,b_0}(0)=b_0 .
 \end{equation}
Let $(\bu_0^n)_n$ be a sequence in  $C^\infty(\R^d)\cap L^2(\R^d ;  \R^N)$ such that 
\begin{equation}
  \label{hyp_0}
\bu_0^n \otimes \bu_0^n 
 \rightharpoonup (b_{x_0,b_0} \otimes b_{x_0,b_0}) \delta_{x_0} \quad   \text{ weak-* in  }  \,  \mathcal M (\R^d ; \R^{N \times N}) , \quad   \text{ as } \, n \rightarrow +\infty .
  \end{equation}
 Then,  
 $$\mathcal C[\bu_0^n]^{-1}   \rightarrow (b_T^\mathcal{N})^*  \, G_T \,  b_T^\mathcal{N}  ,\quad   \text{ as } \, n \rightarrow +\infty .$$
\end{thm}

To prove Theorem \ref{thm00loc}, it is interesting to follow the  view-point introduced at the end of Section \ref{sec-o} and to consider the local energy tensor  $u \otimes u$.
By elementary operations, we deduce from \eqref{toy model} that $u \otimes u$ satisfies the transport PDE: 
\begin{equation}
  \label{tensor-E}
(\partial_t + v \cdot \nabla) (u \otimes u) + A  (u \otimes u) + (u \otimes u) A^* = 0 .
\end{equation}
While we consider some vector fields $u$ with   $L^2(\R^d ; \R^N)$  regularity in space, the equation   \eqref{tensor-E} makes sense  and the Cauchy problem is well-posed, even if 
$u \otimes u$ is a matrix-valued measure $\mu$ (or even more generally for a distribution) rather than a  $L^1(\R^d ; \R^{N \times N})$ function. 
As we are interested in the case where the local energy  tensor  $u \otimes u$ concentrates in  a Dirac mass, it is useful to recall the following elementary result on the Cauchy problem for transport equations with measure solutions.

\begin{Lemma} \label{lem-tr-mu}
Let $n, N \in \N^*$ and $T >0$.
Let $V$ be a regular field of vector field from $\R^n$ to $\R^n$ and $A,\tilde A$ some smooth vector fields of matrices in $\R^{N \times N}$.
 Let $\mathcal M (\R^n ;\R^{N \times N})$ be the space of matrix-valued measures and 
 $$ \mathfrak C := C([0,T]; \mathcal M (\R^n ;\R^{N \times N}) - w* ),$$ 
 the space of the measures depending on time in a continuous way when  $\mathcal M (\R^n ;\R^{N \times N})$ is endowed by its weak-* topology.
Then
\begin{enumerate}
\item For any $\mu_0$ in  $\mathcal M (\R^n ;\R^{N \times N})$, there exists a unique solution 
$\mu \in \mathfrak C,$
to the PDE:
\begin{equation}
  \label{tr-ge}
\partial_t  \mu +  \Div (\mu V ) + A  \mu +\mu \tilde A = 0 ,
\end{equation}
with initial value $\mu_0$. % \quad  \text{ and } \quad 
\item If the initial datum $\mu_0$ is of the form 
 $\mu_0^{y_0,P_0} = P_0 \delta_{y_0}$ with  
 $y_0\in \R^n$ and $P_0 \in \R^{N \times N}$
  then the unique corresponding solution is given by 
\begin{equation}
 \mu^{y_0,P_0}  (t,\cdot) = P(t)  \delta_{y(t)} ,
\end{equation}
where  $y$ solves the nonlinear ODE: 
\begin{equation}
  \label{chch1}
y' = V(y) \quad  \text{ and } \quad y(0)=y_0,
\end{equation}
 and 
$P \in C^1([0,T]; \R^{N \times N} )$ solves the linear ODE: 
\begin{equation}
  \label{chch2}
P' + (A\circ y) P+P (\tilde A \circ  y)=0 \quad  \text{ and } \quad P(0)=P_0.
\end{equation}
\item Finally, if $(\mu_0^n)_n$ a sequence  in $\mathcal M (\R^n ;\R^{N \times N})$ which converges weakly-* to  $\mu_0^{y_0,P_0}$, then the corresponding solution
$\mu^n$  converges to  $\mu^{y_0,P_0} $ in $\mathfrak C$ as $ n \rightarrow +\infty$. 
\end{enumerate}
\end{Lemma}
\begin{proof}
The proof  is very standard and we only give here a few words of comments.
The existence  follows from the characteristics method. 
Indeed solutions can be represented as superpositions of the atomic case considered in the second item. 
The uniqueness and  stability parts of the statement  can be  proved by a duality argument, in the spirit of \cite{DiLi}, or by a quantitative approach, thanks to  matricial Wasserstein distances, see \cites{BV,NG}.
\end{proof}

Equipped with this result, we are now ready to prove Theorem \ref{thm00loc}.

\begin{proof}[Proof of Theorem \ref{thm00loc}]
Let
$X_{x_0} $ be given by \eqref{flow} and $b_{x_0,b_0}$ by \eqref{bxeq}. 
Then, by Leibniz' rule, 
$$
 (b_{x_0,b_0} \otimes b_{x_0,b_0})'  =   -A_{x_0}  (b_{x_0,b_0} \otimes b_{x_0,b_0})  -   (b_{x_0,b_0} \otimes b_{x_0,b_0})   A_{x_0}^*  .
$$
where we recall that $A_{x_0}$ is given by \eqref{coeff-x}. 
Let
$$\mu(t,\cdot) := (b_{x_0,b_0}(t) \otimes b_{x_0,b_0}(t)) \delta_{X_{x_0}(t)} .$$
Therefore, by the second item of Lemma \ref{lem-tr-mu}, 
$\mu$ is the unique solution to 
\begin{equation}
  \label{tensor-E-lim}
\partial_t  \mu +  \Div (\mu V ) + \Big(A- \frac12 (\Div v) \Id \Big) \mu +\mu \Big( A^* - \frac12 (\Div v) \Id \Big)= 0 ,
\end{equation}  
  with the initial data
    $\mu_0 = (b_0 \otimes b_0) \delta_{x_0} $.
\\  \par 

Let  $\bu^n$ be the  solution 
 of the transport PDE \eqref{toy model} associated with the initial data $\bu_0^n$, and set 
 $\mu^n := \bu^n \otimes \bu^n $ the corresponding local energy tensor. Then 
 by  \eqref{tensor-E},
\begin{equation}
  \label{tensor-E-bis}
\partial_t  \mu^n +  \Div (\mu^n V ) + \Big(A- \frac12 (\Div v) \Id \Big) \mu^n +\mu^n \Big( A^* - \frac12 (\Div v) \Id \Big)= 0 .
\end{equation}

Then, by \eqref{hyp_0} and the third item of Lemma \ref{lem-tr-mu}, we obtain that 
$ \mu^n$  converges to  $\mu $ in $\mathfrak C$ as $ n \rightarrow +\infty$. 
\\  \par 

Now, considering again  $M$  given by  \eqref{compo}
 one may  re-use 
\eqref{toy estim}, which can actually be re-obtained as a consequence of the fact that  \eqref{toy multi} is the dual equation of  \eqref{tensor-E} with $B^*B$ as a source term, that is: 
 \begin{align} \label{palab1}
   \int_0^T \|B\bu^n(t)\|_{L^2}^2\,\dd t
 = \int_{\R^d} M(T,x) : \mu^n (T,x) \,\dd x,
 \end{align}
while
 \begin{align} \label{palab2}
  \|\bu^n(T,\cdot)\|_{L^2}^2 
 = \int_{\R^d} \Id : \mu^n (T,x)\,\dd x.
 \end{align}
Thus the proof of Theorem \ref{thm00loc} follows 
by passing to the limit in these two expressions, recalling 
\eqref{compo} for the first one. 
\end{proof}

%%%%%%%%%%%%%%%%%%%%%%%%%%%%%%%%%%%%%%%%%%%%%
\section{Microlocalized cost} 
\label{sec-mc}

As we have seen before, the concept of energy is at the center of the observability issue. In the previous sections, in the case of some transport equations, we considered  the energy tensor $\bu\otimes\bu$, which encodes the localization of the energy in space, with respect to the components of $\bu$. However, for more general evolution linear first order PDEs, in several space dimensions, dispersive effects occur and the propagation of  energy  depends not only on the space variable $x$, but also on the frequency variable $\xi$. Hence we need a tool to encode the distribution of energy in the phase space, generalizing  the energy tensor $\bu\otimes\bu$. To this end, the natural object to consider is the matrix Wigner transform, which we consider in its semiclassical version, focusing on high frequencies. 
 
%%%%%%%%%%%%%%%%%%%%
\subsection{Semiclassical Wigner matrix transform} 

Below $\eps\in (0,1]$. However the regime of interest later on  corresponds to this parameter $\eps$ converging to $0$, its inverse standing for the typical size of the frequencies under consideration. 
\begin{Definition}\label{def:sc-wig}
For $u$ and $v$  in $L^2(\R^d ; \C^N)$,  we define the 
 \textit{Wigner matrix} $W^\eps [u,v](x,y)$, for any $x,\xi\in\R^d\times\R^d$, by
\begin{equation*}
W^\eps [u,v](x,\xi) :=\frac1{(2\pi)^d}\int_{\R^d} u\left(x-\eps \frac y2\right)\otimes \overline v\left(x+\eps \frac y2\right)e^{iy\cdot\xi}\, \dd y.
\end{equation*}
\end{Definition}
The mapping $(u,v) \mapsto W^\eps [u,v]$ is bilinear and verifies.
 \begin{equation}\label{sym wigner}
    W^\eps[u,v]=\left(W^\eps[v,u]\right)^*.
\end{equation}
In the case where $u=v$ we simply denote by $W^\eps [u]$ the Wigner transform $W^\eps [u,u]$, which is a field of symmetric matrices.
The Wigner matrix encodes the energy in the phase space in the following sense: the marginals of $W^\eps [u,v]$  respectively  satisfy, formally, 
 for any $x \in \R^d$, for any $\xi \in \R^d$, 
\begin{align*}
\int_{\R^d}W^\eps [u,v]\, \dd\xi &= u(x) \otimes \overline v(x),\\
\int_{\R^d}W^\eps [u,v]\, \dd x &= \frac1{(2\pi\eps)^d} \widehat u \left(\frac\xi\eps\right) \otimes \left(\widehat v\right)\left(\frac\xi\eps\right),
\end{align*}
where $\hat u, \hat v$ denote the Fourier transforms of $u$ and $v$.
In particular, 
\begin{equation} \label{total-marg}
\iint_{\R^d\times\R^d} \Big( W^\eps[u] : \Id \Big) \, \dd x\dd\xi = \| u \|_{L^2(\R^d; \R^N)}^2  ,
\end{equation}
where we recall that the notation $:$ stands for the Frobenius product.\\
We also have, for any $1\le j,k\le N$, $$\left|W_{j,k}^\eps[u,v](x,\xi)\right|\le \frac1{(\pi\eps)^d} \|u_j\|_{L^2(\R^d;\C)} \|v_k\|_{L^2(\R^d;\C)},$$
where $W_{j,k}^\eps[u,v]$ denotes the coefficient $(j,k)$ of the Wigner matrix, and $u_j$ and $v_k$ respectively denote the $j$-th component of $u$ and the $k$-th component of $v$. This implies that the Wigner matrix $W[u,v]$ is in $L^{\infty}$, from $\R^d\times\R^d$ to $\R^{N\times N}$ equipped with the Frobenius norm, and that $$\|W^\eps[u,v]\|_{L^\infty} \le \frac1{(\pi\eps)^d} \|u\|_{L^2} \|v\|_{L^2}.$$
The Wigner matrix $W^\eps[u,v]$ is also in $L^2(\R^d\times\R^d; \R^{N\times N})$, with $$\|W^\eps[u,v]\|_{L^2}=\|u\|_{L^2}\|v\|_{L^2}.$$ For more details about those properties, we refer, for instance, to \cite{lerner}.

\begin{remark}  The Wigner matrix $W^\eps[u]$ carries information about the energy tensor $u\otimes u$, but is not necessarily nonnegative. However, for positivity, one may consider the Husimi transform instead, or Wigner measures, which are weak limits of Wigner matrices as $\eps$ goes to 0, in some sense. For our purpose, a sign is not required so we will not use those tools. \end{remark}

 Below we recall some classical examples. 
 
  \begin{Example}\label{cs}
An important example is given by the coherent states defined, for any  $x_0,\xi_0 \in\R^d$, for any $\bb_0 \in\R^N$, as follows:
  $$\bu^\eps_{x_0,\xi_0,\bb_0} (x) :=(\pi\eps)^{-\frac d4}\bb_0 e^{i\frac{(x-x_0)\cdot\xi_0}{\eps}}  e^{-\frac{|x-x_0|^2}{2\eps}}.$$
  Then one may compute $$W^\eps[\bu^\eps_{x_0,\xi_0,\bb_0}](x,\xi)=(\pi\eps)^{-d}(\bb_0\otimes\bb_0) e^{-\frac{|x-x_0|^2+|\xi-\xi_0|^2}{\eps}}.$$
  In particular, the matrix Wigner transform 
 $\mathcal W^\eps [\bu^\eps_{x_0,\xi_0,\bb_0}]$ weakly converges, as $\eps \rightarrow 0$,  in the sense of distributions to $(\bb_0 \otimes \bb_0) \delta_{x_0,\xi_0} $.
   \end{Example}

  \begin{Example}
  Another classical example  is given by the WKB states defined,
       for any $\alpha\in(0,1 ]$, any $\xi_0\in\R^d$, and any amplitude function $a(x)\in L^2(\R^d)$,  by
       $$u^\eps_{a,\alpha}(x):=a(x)e^{i\frac{x\cdot\xi_0}{\eps^\alpha}}.$$ 
       Then, as $\eps\to0$, for any $\alpha\in(0,1)$, 
       $$W^\eps[u^\eps_{a,\alpha}]\mathrel{\overset{\makebox[0pt]{\mbox{\normalfont\tiny\sffamily w*}}}{\to}} a(x)\otimes  a(x)\delta_0(\xi)\text{ in }\mathcal S'(\R^d\times\R^d) ,$$ 
     while in the case where $\alpha=1$,
     $$W^\eps[u^\eps_{a,\alpha}]\mathrel{\overset{\makebox[0pt]{\mbox{\normalfont\tiny\sffamily w*}}}{\to}} a(x)\otimes a(x) \delta_{\xi_0}(\xi)  \text{ in }\mathcal S'(\R^d\times\R^d) .$$
  \end{Example}

%%%%%%%%%%%%%%%%%%%%
\subsection{Pseudodifferential operators.}

Our analysis allows to consider not only PDEs but also pseudo-differential equations, what is of interest to cover equations from incompressible fluid mechanics. In microlocal analysis, 
different kinds of quantizations are available, see for instance to \cite[Chapter 4]{zworski}. Here we consider  
  the Weyl quantization, which associates with a symbol $a=a(x,\xi)$ in a suitable class the semiclassical pseudo-differential operator $\w[a]$ 
  defined through its action on functions $u\in L^2(\R^d)$, by
\begin{align*}
( \w[a] u) (x) := (2\pi \eps)^{-d} \int e^{\frac{i(x-y) \cdot \xi}{\eps}} a \left(\frac{x+y}{2}, \xi\right)u(y) \; \dd\xi \, \dd y. 
\end{align*}
Throughout this paper, we will simply assume every symbol to be in the class $S(1)$ of functions bounded as well as all their derivatives, defined by:\begin{equation}\label{def S1}
    S(1):=\{a \in C^\infty(\R^{2d}) : \forall \alpha\in \mathbb N^{2d},\exists C_\alpha>0,\forall (x,\xi)\in\R^{2d}, \|\de^\alpha a\|_\infty\le C_\alpha\}.
\end{equation}
Wider classes of symbols may be studied, but they are not required for the analysis which we perform in this paper. With that choice of class, the pseudo-differential operator given above defines a linear continuous operator from $L^2(\R^d;\C^N)$ to itself, and there exists a constant $C>0$, uniform in $\eps$, as well as a universal constant $\gamma$, such that, for all $a\in S(1),$ it holds 
\begin{equation}\label{norm op}
\|\w[a]\|_{\mathcal L(L^2;L^2)}\le C\sup_{|\alpha|+|\beta|\le \gamma d}\|\de_x^\alpha\de_\xi^\beta a\|_\infty,
\end{equation} 
where $\mathcal L(L^2;L^2)$ denotes the space of linear bounded operators from $L^2$ to itself, endowed with the usual operator norm. From now on, we denote for $k$ in $\N$ by
$\mathcal N_k$ the norm on $S(1)$:\begin{equation}\label{seminorme}
    \mathcal N_k(\cdot):=\sup_{|\alpha|+|\beta|\le k} \|\de_x^\alpha \de_\xi^\beta \cdot\|_\infty.
\end{equation}
We list below some useful properties of calculus of the Weyl quantization. More details, other choices of quantizations and other classes of symbols can be found in \cite{lerner}, \cite{zworski}.
\begin{Lemma}[Adjoint operator.]\label{lem:adj}
    Let $a\in S(1)$ be a symbol. Then the adjoint operator of $\w[a]$ is given by \begin{equation}\label{eq:adjoint}
        (\w[a])^*=\w[a^*].
    \end{equation}
\end{Lemma}

\begin{Lemma}[Composition of operators.]\label{lem:comp}
    Let $a,b\in S(1)$ be two symbols. Then the following composition law holds:\begin{equation}\label{eq:composition}
        \w[a]\w[b]=\w[ab]+\frac\eps{2i} \w[\{a,b\}] + O(\eps^2)_{\mathcal L(L^2;L^2)},
    \end{equation}
    where $\{\cdot,\cdot\}$ denotes the Poisson bracket $$\{a,b\}=(\de_\xi a)(\de_xb)-(\de_xa)(\de_\xi b).$$
\end{Lemma}

The Wigner transform is associated with the semiclassical Weyl quantization by the relation:
 \begin{equation}\label{q-w}
\iint_{\R^d\times\R^d} \Big( W^\eps[u,v] : a(x, \xi) \Big) \, \dd x\dd\xi   = \langle  \w[a^*] u,v \rangle_{L^2} ,
\end{equation}
for any $a(x, \xi)$ in  $S(1)$, and any $u,v\in L^2(\R^d;\C^N)$. 
This property allows us to view Wigner matrices as tempered distributions in $S(1)'$, the dual space of $S(1)$. Indeed, for any $u\in L^2(\R^d;\C^N)$, combining the property \eqref{q-w} with the estimate \eqref{norm op} yields the existence of a constant $C=C(u)>0$, uniform in $\eps$, such that for any $a\in S(1)$ it holds:
$$\left|\iint_{\R^d\times\R^d}W^\eps[u](x,\xi):a(x,\xi)\,\dd x\,\dd\xi\right|\le C\mathcal N_{\gamma d}(a),$$
which allows us to consider $W^\eps[u]$ as a distribution of order at most $\gamma d$ (where $W^\eps[u]=W^\eps[u,u]$ as in \eqref{total-marg}). More precisely, for any $u,v\in L^2(\R^d;\C^N)$, we have $W^\eps[u,v]\in S(1)'_{\gamma d}(\R^{2d})$, where $$S(1)'_k(\R^{2d})=\left\{\psi(x,\xi)\in S(1)' : \exists C>0,\forall \varphi\in S(1), \left|\langle \psi,\varphi\rangle_{\mathcal S',\mathcal S}\right|\le C\mathcal N_k(\varphi)\right\}.$$
We endow $S(1)'_k(\R^{2d})$ with the norm $$\|\psi\|_{\mathcal S'_k}:=\sup_{\phi\in S(1), \mathcal N_k(\phi)\le 1} \langle\psi,\phi\rangle_{S(1)',S(1)}.$$
%
%Lastly, a local version of \eqref{eq:composition} is available.
For any $u,v\in L^2(\R^d ;\C^N)$, and $a\in S(1),$\begin{equation}\label{ww}
    W^\eps[\w[a]u,v]=aW^\eps[u,v]+\frac\eps{2i}\{a, W^\eps[u,v]\}+\eps^2\mathcal R_\eps,
\end{equation} 
where the remainder term $\mathcal R_\eps\in S(1)'_{\gamma d+2}(\R^{2d})$ verifies $$\|\mathcal R_\eps\|_{S(1)'_{\gamma d+2}}=O(1).$$
%
%
%%%%%%%%%%%%%%%%%%%%%%%%%%%%%%%%%%%%%%%%%%%%%%%%%%%%%%%
\subsection{Statement of the result}

With the previous material, it is natural to extend our setting to hyperbolic pseudo-differential equations of the form 
\begin{align}\label{main order diag}
   \eps \de_t\bu+\w[i H(x, \xi) \Id + \eps A(x, \xi)]\bu =0,\quad \text{on } [0,T], 
\end{align}
where the principal part $H$ is a real scalar symbol in $S(1)$ 
and the sub-principal part $A$ is a matrix-valued symbol in $S(1)$.
\par \ 

Our setting also allows to consider a wider class of  observability operators  given as well by  pseudodifferential operators associated with a symbol matrix $B$ in $S(1)$. Thus we are led to the following extension of the notion of observability cost $\mathcal C[\bu_0]$ of an initial data  $\bu_0$ at time $T$ for the equation  \eqref{main order diag}  
by the pseudodifferential observability operator of symbol $B$ in $S(1)$.

\begin{Definition}\label{def:mcost}
Let $T>0$. 
Let  $(x_0 ,\xi_0,b_0) \in \R^d \times (\R^d\setminus\{0\})  \times \R^N$ and
 $\bu_0$  in $L^2(\R^d ; \R^N)$. 
The observability cost  $\mathcal C[\bu_0]$   of the initial data $\bu_0$ at time $T$ for the equation  \eqref{main order diag}  
by the observability operator of symbol $B$ in $S(1)$
is defined by 
 \begin{equation}
  \label{costuf}
 \mathcal C[\bu_0] := 
  \frac{     \|\bu(T,\cdot)\|_{L^2}^2 }{   \int_0^T\|\w[B]\bu(t)\|_{L^2}^2\,\dd t  } ,
  \end{equation}
 where $\bu$ is the  solution 
 of  \eqref{main order diag}  associated with $\bu_0$. 
\end{Definition}

The purpose of this section is to study  the observability cost 
of  initial data which are localized  at $(x_0 ,\xi_0 ,\bb)$ in the following sense. 
\begin{Definition}\label{def:loc}
A family $(\bu_0^\varepsilon)_{\eps \in (0,1]}$  in $L^2(\R^d ; \R^N)$ is said to be localized at 
$(x_0 ,\xi_0,b_0) \in \R^d \times (\R^d\setminus\{0\})  \times \mathcal S^{N-1}$ if its 
 \textit{Wigner matrix}  satisfies 
 $$\mathcal W^\eps  [\bu_0^\varepsilon]  \mathrel{\overset{\makebox[0pt]{\mbox{\normalfont\tiny\sffamily w*}}}{\to}} (b_0 \otimes b_0) \delta_{x_0 ,\xi_0} \quad   \text{  in  }  \,  \mathcal (S(1))', \quad   \text{ as } \, \eps \rightarrow 0 .$$
\end{Definition}
The question at stake is then the following. 
\begin{Question} \label{qul}
Let $T>0$. 
Let  $(x_0 ,\xi_0,b_0) \in \R^d \times (\R^d\setminus\{0\} ) \times \R^N$ and
a family $(\bu_0^\varepsilon)_{\eps \in (0,1]}$  in $L^2(\R^d ; \R^N)$  localized at $(x_0 ,\xi_0,b_0)$. 
Can we determine the limit of 
 the observability cost  
$\mathcal C[\bu_0^\varepsilon]$ 
 as $\eps \rightarrow 0$ ?
\end{Question}

The  result below provides an answer  in terms of  the Hamiltonian flow induced by the principal symbol $H$, the microlocalized observability Gramian matrices, the bicharacteristic equations and the  bicharacteristic amplitude equation whose definitions are now given. 
\begin{Definition}\label{flotH}
Let  $H$ be a real scalar symbol in $S(1)$ and $v_{H} $ the divergence free vector field from $\R^d \times \R^d$ to $\R^d \times \R^d$  defined by 
\begin{equation}
\label{ham-field}
v_{H} 
:= (\nabla_{\xi} H , -\nabla_{x} H).
\end{equation}
For any  $x_0 ,\xi_0 \in\R^d$, 
we define the Hamiltonian flow  $(X_{x_0 ,\xi_0 }(t), \Xi_{x_0 ,\xi_0 }(t))$ 
induced by $H$ as the unique solution to the nonlinear ODE:  
\begin{equation}
\label{bich1}
\big( X_{x_0 ,\xi_0 }, \Xi_{x_0 ,\xi_0 } \big)' 
= v_{H}  \big( (X_{x, \xi}, \Xi_{x_0, \xi_0} )\big), \quad \left(X_{x_0 ,\xi_0 },\Xi_{x_0 ,\xi_0 }\right)(0)=(x_0 ,\xi_0 ).
\end{equation}
\end{Definition}
 
  \begin{remark} \label{Liouville}
    It is worth noticing that the determinant 
    $J(t):=\det(D(X_{x_0,\xi_0}(t), \Xi_{x_0,\xi_0}(t)))$  of the Jacobian matrix associated with the Hamiltonian flow satisfies
     $J(t)=1$ for all $t\in [0,T]$ since 
$   J'(t)=\mdiv \left(\nabla_\xi H, -\nabla_x H\right) J(t)=0$ 
and $J(0)=0$. 
Here and below we use the notation $\mdiv$ to denote the divergence operator in $\R^d \times \R^d$.
\end{remark}
We now define the microlocalized observability Gramian matrices, as solutions 
 of some differential Lyapunov equations involving the coefficients of \eqref{main order diag} along the Hamiltonian flow. 
\begin{Definition}\label{flot-Gr}
For any  $x_0 ,\xi_0 \in\R^d$, we define the Gramian matrix microlocalized in $(x_0 ,\xi_0 )$  as the unique solution $G_{x_0 ,\xi_0 }$ to the differential  Lyapunov equation $Lyap(A_{x_0 ,\xi_0 },B_{x_0 ,\xi_0 })$ where the time-dependent matrices  
  $A_{x_0 ,\xi_0 }$ and $B_{x_0 ,\xi_0 }$ are given by
$$A_{x_0 ,\xi_0 }(t):=A(X_{x_0, \xi_0} (t),\Xi_{x_0 ,\xi_0 }(t)) \quad   \text{ and }   \quad   B_{x_0 ,\xi_0 }(t):=B(X_{x_0, \xi_0} (t),\Xi_{x_0 ,\xi_0 }(t)).$$
\end{Definition}
Finally we define the bicharacteristic  amplitude.
\begin{Definition}\label{biflotH}
For any  $x_0,\xi_0\in\R^d$, for any $\bb_0$ in $\R^N$, 
we define 
 the bicharacteristic  amplitude $\bb_{x_0,\xi_0,\bb_0}(t)$ associated with $H$ and $A$ 
 as the unique solution to the linear ODE:
  \begin{equation}
  \label{bich3}
\bb_{x_0,\xi_0,\bb_0}'=-A_{x_0,\xi_0} \,  \bb_{x_0,\xi_0,\bb_0},\quad \bb_{x_0 ,\xi_0 ,\bb_0}(0)=\bb_0 .
\end{equation}
\end{Definition}

The  result below answers to Question \ref{qul} and identifies the limit of 
 the observability cost  
$\mathcal C[\bu_0^\varepsilon]$ of a family $(\bu_0^\varepsilon)_{\eps \in (0,1]}$  in $L^2(\R^d ; \R^N)$  localized at $(x_0 ,\xi_0,b_0)$ as $\eps \rightarrow 0$ in terms of 
the Rayleigh quotient of the  microlocalized observability Gramian matrix  along the bicharacteristic  flow with   the  bicharacteristic amplitude, at the final time.

\begin{thm}
\label{PDO transport}
Let $T>0$. 
Let $x_0\in \R^d, \xi_0\in \R^d\setminus\{0\}$ and set 
 $G_T := G_{x_0,\xi_0} (T),$
 where $G_{x_0,\xi_0}$ is  the Gramian matrix microlocalized in $(x_0,\xi_0)$ as given by Definition \ref{flot-Gr}.
Let $\bb_0 \in  \mathcal S^{N-1}$ and set 
$$\bb_T^\mathcal{N} :=\frac{\bb_{x_0,\xi_0,\bb_0} (T)}{|\bb_{x_0,\xi_0,\bb_0} (T)|},$$ where $\bb_{x_0,\xi_0,\bb_0} $ is the transported amplitude associated with $H$ and $A$ given by  
Definition \ref{biflotH}. 
Let $(\bu_0^\varepsilon)_{\eps \in (0,1]}$  in $C^\infty(\R^d)\cap L^2(\R^d ; \R^N)$ be localized at $(x_0 ,\xi_0,b_0)$ in the sense of Definition \ref{def:loc}.
Then  the observability cost $\mathcal C[\bu_0^\eps]$ satisfies
$$\mathcal C[\bu_0^\eps]^{-1} \rightarrow  (\bb_T^{\mathcal{N} })^* \, G_T \, \bb_T^{\mathcal{N} } ,\quad   \text{ as } \,  \eps \rightarrow 0 .$$
\end{thm}
Let us highlight that the interest of the result above is that the cost of observability of initial data which are localized at  $(x_0 ,\xi_0,b_0)$ converges to a limit which can determined by ODEs. 
It can be considered as a broad extension of the result obtained in the case of transport equations in  Theorem \ref{thm00loc}, to general hyperbolic pseudo-differential equations of the form  \eqref{main order diag}, in the high frequency limit, despite the extra dispersive effects. As a matter of fact, 
 the equation \eqref{toy model} can be recast as an equation of the form \eqref{main order diag}
%\begin{align}   \eps \de_t\bu+ \w[i (v(x) \cdot \xi) \Id + \eps A(x, \xi)]\bu =0,\quad \text{on } [0,T], \end{align}
% 
with
\begin{equation}\label{toB}
H(x, \xi) = v(x) \cdot \xi \quad \text{and} \quad A (x,\xi) :=  A (x)  -   \frac12  (\Div v(x))  \Id.
\end{equation}
Then the  the bicharacteristic equations \eqref{bich1} reduce to 
\begin{gather}
\label{bich1-transport}
X'_{x_0,\xi_0}(t)= v(X_{x_0, \xi_0} (t)), 
 \,  \text{ and }  \,
\Xi'_{x_0,\xi_0}(t)= -(\nabla_x v(X_{x_0, \xi_0} (t)))^* \Xi_{x_0,\xi_0}(t) \quad \text{for} \; t\in  (0,T) .
\end{gather}
  In particular, the first equation does not depend on $\xi$, translating the lack of dispersion in the case of  transport equations. 
  Moreover the matrix $A (x,\xi)$ defined in \eqref{toB} is the one which is involved in the Lyapounov equation satisfied by the Gramian matrix $ G_{x_0}$, see 
  \eqref{coeff-x}.

%%%%%%%%%%
\subsection{Proof of Theorem \ref{PDO transport}}

For any $(x,\xi)$, 
  let $G_{x,\xi}$ is  the Gramian matrix microlocalized in $(x,\xi)$ as given by Definition \ref{flot-Gr}.
We define the symbolic multiplier $M(t,x,\xi)$
  by setting, for any $(t,x,\xi)$, 
  \begin{equation} \label{eq:imply}
      G_{x,\xi}(t)=:M(t,(X_{x, \xi} (t),\Xi_{x,\xi}(t))).
  \end{equation}
 %$$$$
  %
  As the Hamiltonian flow is invertible, see  Remark \ref{Liouville}, the above $M$ is uniquely defined.
Similarly to the reformulation in  \eqref{palab1} and  \eqref{palab1} performed in the non-dispersive case in 
 Section  \ref{sec-local}, we are going to use a reformulation of the observability cost in terms of 
 a  recorded energy by the use of a multiplier. The latter is precisely chosen as  the symbolic multiplier $M(t,x,\xi)$, and the corresponding recorded energy
is then set as
 \begin{align}\label{eq:energy-rec}
    \iint_{\R^d\times\R^d}\mathcal W^\eps(t,x,\xi):M(t,x,\xi)\,\dd x\dd\xi=\langle\w[M(t,x,\xi)]\bu,\bu\rangle_{L^2(\R^d)} ,
\end{align}
 where we use the shorthand notation $\mathcal W^\eps$ for the Wigner matrix transform $\mathcal W^\eps  [\bu^\eps]$, recalling the property \eqref{q-w}. \par \ 

 Let $T>0$ and $(\bu_0^\varepsilon)_{\eps \in (0,1]}$ be bounded in $L^2(\R^d ; \R^N)$. 
A key point of the proof of Theorem \ref{PDO transport} is to establish that 
as $ \eps \rightarrow 0 $,  the observability cost $\mathcal C[\bu_0^\eps]$ satisfies 
\begin{equation}
  \label{ratio}
\mathcal C[\bu_0^\eps]=\frac{\int_{\R^d\times\R^d} \mathcal W^\eps (T,\cdot):\Id\,\dd x\,\dd\xi}{\int_{\R^d\times\R^d}\mathcal W^\eps(T,\cdot):M(T,\cdot) \,\dd x\,\dd\xi}+o(1).
\end{equation}
Recalling \eqref{total-marg}, we already have: 
    \begin{align}\label{eq-est2}
     \|\bu(T,\cdot)\|_{L^2}^2
&=  \iint_{\R^d\times\R^d}  \mathcal W^\eps (T,\cdot)  : \Id \, \dd x\dd\xi   .
\end{align}
The reformulation of the denominator is more subtle. 
First, by \eqref{q-w}, the properties of the Frobenius product and \eqref{ww}, we observe that, at any time, 
\begin{align}
\|\w[B] \bu\|_{L^2}^2 =
\iint_{\R^d\times\R^d} \Big(\mathcal W^\eps  :  B^*B \Big) \, \dd x\dd\xi 
%=\left(\iint_{\R^d\times\R^d} (W^\eps[\w[B]\bu(t)])\dd x\,\dd\xi\right)(t)+O(\eps)=\|\w[B] \bu(t,\cdot)\|_{L^2}^2
+O(\eps).
\end{align}
Then, using the chain rule  and \eqref{bich1},  we observe that 
\begin{align*}
  G'_{x,\xi} (t)= \Big( \de_t M +\mdiv(Mv_H) \Big)(t,X_{x, \xi} (t) , \Xi_{x, \xi}(t) ) ,
\end{align*}
so that, since $G_{x,\xi}$ verifies the differential  Lyapunov equation $Lyap(A_{x,\xi},B_{x,\xi})$, we infer that the symbolic multiplier $M$  satisfies the following PDE:
\begin{equation}
\label{mul-eq}
    \left\{
    \begin{aligned}
    \de_tM+\mdiv(Mv_H)-MA-A^*M&=B^*B,\\
    M(0)&=0,
    \end{aligned}
    \right.
\end{equation}
 where $B^*B$ appears as a source term.
On the other hand, the dynamics of the Wigner matrix $\mathcal W^\eps $ is given by the following result. 
\begin{Lemma}
\label{tr-wi}
%Let  $\bu$ the solution  to \eqref{main order diag}.
There exists a remainder term $\mathcal R_\eps$ in $ (S(1))_{\gamma d+2}'$ with $\|\mathcal R_\eps\|_{\mathcal S'_{\gamma d+2}}=O(\eps)$ such that 
\begin{equation}
  \label{eq-p-w}
\de_t\mathcal W^\eps +  \mdiv (\mathcal W^\eps  v_{H} ) % \{H\Id,\mathcal W^\eps\}
 + A\mathcal W^\eps + \mathcal W^\eps A^* =\mathcal R_\eps.
\end{equation}
\end{Lemma}
\begin{proof}
By bilinearity and the symmetry property \eqref{sym wigner},  the Wigner matrix transform  $\mathcal W^\eps$ satisfies 
 \begin{align}\label{deriv2}
\de_t  \mathcal W^\eps &= %W^\eps(\de_t\bu,\bu)+W^\eps(\bu,\de_t\bu)=
    W^\eps[\de_t\bu,\bu]+W^\eps[\de_t\bu,\bu]^*.
    \end{align}
 Moreover, by \eqref{main order diag}, it follows that
\begin{align*}
    W^\eps[\de_t\bu,\bu]
    &=-W^\eps\left[\w\left[\frac i\eps H(x, \xi)\Id + A(x, \xi)\right]\bu, \bu \right],
\end{align*}
so that, using the property \eqref{ww}, we arrive at 
 $$W^\eps[\de_t\bu,\bu]=-\frac i\eps H \mathcal W^\eps - \frac12\{H,\mathcal W^\eps\}-A\mathcal W^\eps+O_{(S(1))_{\gamma d+2}'}(\eps).$$
Substituting into \eqref{deriv2} and using that $H$ is a real scalar symbol,   we obtain 
\begin{equation}
  \label{eq-p-w2}
\de_t\mathcal W^\eps + \{H\Id,\mathcal W^\eps\} + A\mathcal W^\eps + \mathcal W^\eps A^* = O_{(S(1))_{\gamma d+2}'}(\eps).
\end{equation}
Recalling that   $v_{H} $ is the divergence free vector field  defined by \eqref{ham-field} the conclusion follows. 
%Up to an error term, it is the Liouville equation associated with the Hamiltonian flow induced by $H$, with some amplification/attenuation due to the matrix $A$.\
%One can notice that the Wigner matrix $\mathcal W^\eps$ is also transported along the Hamiltonian flow induced by $H$, and is amplified by the matrix $-A^*$ along the flow.\\     
\end{proof}

As in the proof of Theorem \ref{thm00loc} we rely on a duality argument given by the following result.
\begin{Lemma}
\label{tr-dua-stab}
Let $W(t,x,\xi)$ and $M(t,x,\xi)$ two matrix-valued time-dependent symbols such that \begin{equation*}\left\{\begin{aligned}\de_tW+\mdiv(Wv_H)+AW+WA^*&=\mathcal R,\\
\de_tM+\mdiv(Mv_H)-MA-A^*M&=\mathcal F.\end{aligned}\right.\end{equation*}
Then
    \begin{align}
&\Big( \iint_{\R^d\times\R^d} \Big( W : M \Big) \, \dd x\dd\xi \Big) (T)
=
\Big( \iint_{\R^d\times\R^d} \Big( W : M \Big) \, \dd x\dd\xi \Big)(0)
 \\\nonumber \quad  \quad \quad &+
\int_0^T 
\Big( \iint_{\R^d\times\R^d} \Big(\mathcal R : M+ W : \mathcal F \Big) \, \dd x\dd\xi 
\Big)(t) \, dt
 .
\end{align}
\end{Lemma}
% 
%\subsection{Duality. Proof of Lemma \ref{tr-dua-stab}}
\begin{proof}
 We observe that the time derivative $E'$ of 
$$E(t):= \iint_{\R^d\times\R^d} M:W \,\dd x\,\dd\xi, $$
satisfies
\begin{align}\notag E'(t)=&\iint_{\R^d\times\R^d}\left(-\mdiv(Mv_H)+MA+A^*M+\mathcal F\right):W\,\dd x\,\dd\xi\\&+\iint_{\R^d\times\R^d}M:\left(-\mdiv(Wv_H)-AW-WA^*+\mathcal R\right)\,\dd x\,\dd\xi.\label{w:m}\end{align}
Using that $\mdiv(v_H)=0$, an integration by parts yields 
$$\iint_{\R^d\times\R^d} \mdiv(Mv_H):W\,\dd x\,\dd\xi=-\iint_{\R^d\times\R^d}M:\mdiv(Wv_H)\,\dd x\,\dd\xi.$$
For the other terms, simple computations give
$$(MA+A^*M):W=M:(AW+WA^*).$$
Using those two equalities in \eqref{w:m} leads to \begin{equation}
    E'(t)=\iint_{\R^d\times\R^d}(\mathcal F: W + M:\mathcal R)\,\dd x\,\dd\xi.
\end{equation}
Integrating this equality in time, from $0$ to $T$, we conclude the proof of Lemma \ref{tr-dua-stab}.
\end{proof}
We now  apply Lemma \ref{tr-dua-stab} with $W=\mathcal W^\eps$ and $M$, which yields, thanks to \eqref{mul-eq} and Lemma \ref{tr-wi}, 
    \begin{align}\label{eq-est1}
&\Big( \iint_{\R^d\times\R^d} \Big( \mathcal W^\eps  : M \Big) \, \dd x\dd\xi \Big) (T)
= 
\int_0^T  \iint_{\R^d\times\R^d} \Big(\mathcal W^\eps : B^*B \Big) \, \dd x\dd\xi \Big)(t) \, dt
%\Big( \iint_{\R^d\times\R^d} \Big(\mathcal R : M \Big) \, \dd x\dd\xi 
+ O(\eps).
\end{align}
From \eqref{eq-est2} and \eqref{eq-est1},   we deduce 
  \eqref{ratio}. \ \par \

  Now let $x_0\in \R^d, \xi_0\in \R^d\setminus\{0\}$ and $\bb_0 \in  \mathcal S^{N-1}$ and 
  assume that $(\bu_0^\varepsilon)_{\eps \in (0,1]}$  is localized at $(x_0 ,\xi_0,b_0)$ in the sense of Definition \ref{def:loc}.
 To pass to the limit in  \eqref{ratio}
   we rely on the following intermediate result. 
\begin{Lemma} \label{tr-mu-wig}
Let $n, N \in \N^*$ and $T >0$.
Let $V$ be a regular field of vector fields from $\R^d\times\R^d$ to itself, $A,\tilde A\in S(1)$ be two smooth vector fields of matrices in $\R^{N \times N}$, and $(\mathcal R_\eps)_{\eps\in(0,1]}\subset (S(1))'_{\gamma d+2}$ such that \begin{equation}\label{lim Re}\lim_{\eps\to 0}\|\mathcal R_\eps(t)\|_{\mathcal S'_{\gamma d+2}}=0,\text{ uniformly in time.}\end{equation}
 Let 
 $$ \mathfrak C := C([0,T]; (S(1))' - w* ),$$ 
 the space of the distributions in the dual space to $S(1)$, depending on time in a continuous way when  $(S(1))'$ is endowed by its weak-* topology.
Then
\begin{enumerate}
\item For any $\mu_0$ in  $(S(1))'$, $\eps$ in $(0,1]$, there exists a unique solution 
$\mu \in \mathfrak C,$
to the PDE:
\begin{equation}
\partial_t  \mu +  \mdiv (\mu V ) + A  \mu +\mu \tilde A = \mathcal R_\eps ,
\end{equation}
with initial value $\mu_0$. % \quad  \text{ and } \quad 
\item If the initial data $\mu_0$ is of the form 
 $\mu_0^{x_0,\xi_0,P_0} = P_0 \delta_{x_0,\xi_0}$ with  
 $x_0,\xi_0\in \R^d$ and $P_0 \in \R^{N \times N}$
  then the unique corresponding solution to the homogeneous equation is given by 
\begin{equation}
  \label{tr-for}
 \mu^{x_0,\xi_0,P_0}  (t,\cdot) = P(t)  \delta_{x(t),\xi(t)} ,
\end{equation}
where  $(x,\xi)$ solves the nonlinear ODE: 
\begin{equation}
(x,\xi)' = V(x,\xi) \quad  \text{ and } \quad (x,\xi)(0)=(x_0,\xi_0),
\end{equation}
 and 
$P$ the linear ODE: 
\begin{equation}
P' + A(x(\cdot),\xi(\cdot))P+P \tilde A(x(\cdot),\xi(\cdot))=0 \quad  \text{ and } \quad P(0)=P_0.
\end{equation}
\item Finally, if $(\mu_0^\eps)_{\eps\in(0,1]}$ converges weakly-* to  $\mu_0^{x_0,\xi_0,P_0}$ in $(S(1))'$ as $\eps$ goes to 0, then the corresponding solution 
$(\mu^\eps)_{(0,1]}$  converges to  $\mu^{x_0,\xi_0,M_0} $ in $\mathfrak C$.
\end{enumerate}
\end{Lemma}
 \begin{proof}
The existence of the solution, as well as the form of the solution associated with the initial datum $M_0\delta_{x_0,\xi_0}$, follows from the characteristics method. The uniqueness and stability can be proved by a standard duality argument.
 \end{proof}
By the second item of Lemma \ref{tr-mu-wig}, 
for any $(x_0,\xi_0,\bb_0)\in \R^{d}  \times \R^{d} \times \R^N$, 
 the unique solution $\mu^{x_0,\xi_0,b_0}$ to  %\eqref{t-wig} 
\begin{equation}
\label{t-wig-re}
\de_t \mathcal \mu^{x_0,\xi_0,b_0} + \mdiv ( \mathcal \mu^{x_0,\xi_0,b_0} v_{H} ) + A\mathcal \mu^{x_0,\xi_0,b_0} + \mathcal  \mu\mu^{x_0,\xi_0,b_0} A^* = 0 .
  \end{equation}
        associated with the initial data
    $\mu_0^{x_0,\xi_0,b_0} := (\bb_0 \otimes \bb_0) \delta_{x_0,\xi_0} $
    is 
    $$\mu^{x_0,\xi_0,b_0} = (\bb_{x_0,\xi_0,b_0}(t) \otimes \bb_{x_0,\xi_0,b_0}(t))  \delta_{X_{x_0 ,\xi_0 }(t), \Xi_{x_0 ,\xi_0 }(t)} ,$$
    where we recall that  the Hamiltonian flow  $(X_{x_0 ,\xi_0 }(t), \Xi_{x_0 ,\xi_0 }(t))$ is defined by  the 
 equations \eqref{bich1}  and  \eqref{bich3}.
Moreover,  by the third item of Lemma \ref{tr-mu-wig}, since, as  $\eps \rightarrow 0 $,
 $\mathcal W^\eps  \mathrel{\overset{\makebox[0pt]{\mbox{\normalfont\tiny\sffamily w*}}}{\to}} (b_0 \otimes b_0) \delta_{x_0 ,\xi_0} $   in $ \mathcal (S(1))'$,  
 we have that, as $\eps \rightarrow 0 $,
\begin{equation}\label{conv-W}\mathcal W^\eps(T)\mathrel{\overset{\makebox[0pt]{\mbox{\normalfont\tiny\sffamily w*}}}{\to}} (\bb_{x_0,\xi_0,\bb_0}(t)\otimes\bb_{x_0,\xi_0,\bb_0}(T))\delta_{X_{x_0,\xi_0}(T),\Xi_{x_0,\xi_0}(T)}.\end{equation}
Passing to the limit in  \eqref{ratio} and recalling \eqref{eq:imply} 
we conclude the proof of Theorem \ref{PDO transport}.

%%%%%%%%%%%%%%%%%%%%%%%%%%%%%%

\section{Application to the incompressible Euler equations linearized at a Couette flow}\label{sec:example-euler}
%Application to the incompressible Euler equations linearized at a Couette flow in velocity form for a  constant non-invertible observability matrix

An interesting example of application of our method is represented by the incompressible Euler equations in $\R^2$, in velocity form, which we linearize at a steady state $\overline{\bu}$ satisfying  $\Div(\overline\bu)=0$, so that the velocity field $\bu=\bu(t,x) \in \R^N$ and the  pressure $p=p(t,x) \in \R$ satisfy 
\begin{equation}\label{incomp euler}\left\{\begin{aligned}\de_t\bu+(\overline{\bu}\cdot\nabla)\bu+(\bu\cdot\nabla)\overline{\bu} + \nabla p&=0\\\Div(\bu)&=0.\end{aligned}\right.\end{equation}
Using  the Leray projection onto divergence-free vector fields,
the semiclassical approximation of \eqref{incomp euler} then reads 
\begin{equation}\label{euler fourier}\eps \de_t\bu + \w[i H (x,\xi)\Id + \eps A(x, \xi)]\bu = 0,\end{equation} 
where %$r\in S^{-1}$ is a remainder term, while 
\begin{equation}\label{ssy}
H (x,\xi)=\overline\bu(x)\cdot\xi, \quad A(x,\xi)=\left(\Id - 2 \frac{\xi\otimes\xi}{|\xi|^2}\right) D\overline{\bu}.
\end{equation} 
Since we want to study the behavior of solutions that are localized, and in high-frequency, we may multiply the symbols $A$ and $H$ by a cut-off function that is equal to $0$ near $\xi=0$ to avoid the singularity of $A$ at $\xi=0$, so that one may   consider a symbol $A$ in $S(1)$ which is given by \eqref{ssy} for $\xi=0$ outside of a ball of radius say $1/2$. \ \par \ 

We consider the case where the steady state $\overline{\bu}$ is the Couette flow $\overline{\bu}=(x_2, 0)$, and where 
 the observation matrix is 
  $$B=\begin{pmatrix}1&0\\0&0\end{pmatrix}.$$
  
  Then for any $x_0,\xi_0$ in $\R^2$, the Gramian matrix $G_{x_0,\xi_0}$ microlocalized at  $(x_0,\xi_0)$ is given by  
  $$G_{x_0,\xi_0}(T)= \int_0^T(R_{x_0,\xi_0}(t,T)^*B^*BR_{x_0,\xi_0}(t,T))\, \dd t,$$
  where $R_{x_0,\xi_0}(t,T)$ is the resolvent matrix, given by solving the linear system
   $$\frac{\dd}{\dd t}R_{x_0,\xi_0}(t,T)=-A(X_{x_0,\xi_0}(t), \Xi_{x_0,\xi_0}(t)) R_{x_0,\xi_0}(t,T), \quad R(T,T)=\Id.$$

A tedious computation yields, with the notation $\xi_0=\left(\xi_{1,0}, \xi_{2,0}\right)$, 
 $$R_{x_0,\xi_0}(t,T)=\begin{pmatrix}1&f(t,\xi_0)\\0&g(t,\xi_0)\end{pmatrix},$$
where
\begin{equation}\label{eq:f}\begin{aligned}
f(t,\xi_0)=\begin{cases}T-t &\text{if }\xi_{1,0}=0,\\-\frac{(\xi_{2,0}-t\xi_{1,0})(\xi_{1,0}^2+(\xi_{2,0}-T\xi_{1,0})^2)}{\xi_{1,0}(\xi_{1,0}^2+(\xi_{2,0}-t\xi_{1,0})^2)}+\frac{\xi_{2,0}-T\xi_{1,0}}{\xi_{1,0}}&\text{otherwise,}\end{cases}\end{aligned}
\end{equation}
 and $$g(t,\xi_0)=\begin{cases}1&\text{if }\xi_{1,0}=0,\\\frac{\xi_{1,0}^2 + (\xi_{2,0}-T\xi_{1,0})^2}{\xi_{1,0}^2+(\xi_{2,0}-t\xi_{1,0})^2}&\text{otherwise.} \end{cases}.
$$
Moreover 
 $$G_{x_0,\xi_0}(T)=\int_0^T \begin{pmatrix}1&f(t,\xi_0)\\ f(t,\xi_0)&f^2(t,\xi_0)\end{pmatrix}\dd t.$$
 We now consider $(x_0,\xi_0,\bb_0)\in \R^2\times(\R^2\setminus\{0\})\times\mathcal S^1$  and  $(\bu_0^\eps)_{\eps\in(0,1]}\subset L^2(\R^2;\R^2)$ a family of initial data localized in $(x_0, \xi_0, \bb_0).$ 
By Theorem \ref{PDO transport}, 
 the observability cost $\mathcal C [\bu_0^\eps]$ satisfies,  as $\eps$ goes to $0$, 
 $$\lim_{\eps\to0}(\cC[\bu_0^\eps])^{-1}=(\bb_T^{\mathcal N})^*G_{x_0,\xi_0}(T)\bb_T^{\mathcal  N},$$with $$\bb_T^{\mathcal N}=\frac{\bb_T}{|\bb_T|},\quad \bb_T=R_{x_0,\xi_0}(T,0)\bb_0=\begin{pmatrix}1&-\frac{f(0,\xi_0)}{g(0,\xi_0)}\\0&\frac1{g(0,\xi_0)}\end{pmatrix}\bb_0.$$
We can notice that, for any value of $\xi_0\ne0$,  the function $f(\cdot,\xi_0)$ is continuous and not constant, so that, by the Cauchy-Schwarz inequality,
\begin{align*}
\det G_{x_0,\xi_0}(T) = T\int_0^T f^2(t, \xi_0) \, \dd t -  \left(\int_0^T f(t, \xi_0) \, \dd t\right)^2 > 0.
\end{align*}
Thus the system is observable in high frequencies for initial data localized in any $(x_0,\xi_0,\bb_0)$, and for any positive time $T>0$. The limit of the observability cost as $\eps$ goes to $0$ is finite, and can be explicitly computed in terms of $\xi_0$ and $\bb_0.$
The short-time and long-time behaviors of the observability cost can be given by further computations.
First, in the case where  $\xi_{1,0}=0$, 
  the Gramian matrix has the simple expression:
\begin{equation*}
    G_{x_0,\xi_0}(T)=\begin{pmatrix}
        T&\frac{T^2}2\\\frac{T^2}2&\frac{T^3}3
    \end{pmatrix}.
\end{equation*}
Immediate computations then yield 
$$\lambda_{\rm{min}}(G_{x_0,\xi_0}(T))=\frac T2+\frac16(T^3-\sqrt{T^6+3T^4+9T^2}),$$ 
so that in particular 
  $\lambda_{\rm{min}}(G_{x_0,\xi_0}(T))^{-1} \sim\frac{12}{T^3}$ for short times 
  and $\lambda_{\rm{min}}(G_{x_0,\xi_0}(T))^{-1} \sim \frac4T$ for long times. 
  Note that $\lambda_{\rm{min}}(G_{x_0,\xi_0}(T))^{-1}$ is the "worst-case" cost for initial data localized in $(x_0,\xi_0)$, namely $$\lambda_{\rm{min}}(G_{x_0,\xi_0}(T))^{-1}=\sup_{\bb_0\in \mathcal S^{N-1}} \{\lim_{\eps\to0}\mathcal C(\bu_0^\eps) : (\bu_0^\eps) \text{ is localized in }(x_0,\xi_0,\bb_0)\}.$$
  The asymptotic behaviour of the cost at 
 short times can be extended to the other values of $\xi_0$. 
 If $\xi_{1,0}\ne0$ and $\xi_{1,0}^2\ne \xi_{2,0}^2$, then by using the Taylor expansion of $f$ at first order, we obtain 
\begin{equation}
    \label{-3}
    \lambda_{\rm{min}}(G_{x_0,\xi_0}(T)) \sim \frac{T^3(\xi_{1,0}^2-\xi_{2,0}^2)^2}{12|\xi_0|^4} .
\end{equation}
 Otherwise, that is if  $\xi_{1,0} \neq 0$ and $\xi_{1,0}^2=\xi_{2,0}^2$, then 
  by using the Taylor expansion of $f$ at second order, we obtain 
  \begin{equation}
    \label{-5}
  \lambda_{\rm{min}}( G_{x_0,\xi_0}(T)) \sim \frac{T^5}{45} .
  \end{equation}
Those estimates are consistent with the results of \cite{seidman1}. Indeed, in the case where $\xi_{1,0}^2\ne\xi_{2,0}^2$, the rank condition is fulfilled on the second bracket, namely $$\text{rank}[B|A_{x_0,\xi_0}(0)B]=2, \text{ while  rank}(B)<2.$$
Then the analysis of \cite{seidman1} gives a cost of order $T^{-3}$, as given by   \eqref{-3}. When $\xi_{1,0}^2=\xi_{2,0}^2$, we need to take one more bracket to verify the rank condition, namely $$\text{rank}[B|A_{x_0,\xi_0}(0)B|(A_{x_0,\xi_0}^2(0)+A'_{x_0,\xi_0}(0))B]=2,\text{ while  rank}[B|A_{x_0,\xi_0}(0)B]<2.$$ The results in \cite{seidman1} then yield a cost of order $T^{-5}$, as given by \eqref{-5}. Our approach furthermore gives an explicit constant.

%%%%%%%%%%%%%%%%%%%%
%%%%%%%%%%%%%%%%%%%%
\section{Case of  systems without crossing modes}

This section is devoted to an extension of the results of Section 
\ref{sec-mc} to more general systems. 

\subsection{Statement of the result}
%\subsection{Statement of the result}

In this section we extend the previous analysis to the case of pseudo-differential hyperbolic systems of the form: 
\begin{align}\label{hyper-eps}
\de_t \bu + \w\left[\frac i\eps H + A\right]\bu=0, \quad (x, \xi) \in \R^d\times \R^d .
\end{align}
Above, $A$ and $H$ are two matrices of symbols which are supposed to fulfill the following assumption. 
%with the two following hypotheses.
%
  \begin{assumption}\label{hyp-sys} 
\begin{itemize} We assume that $A$ and $H$ satisfy the following.
\item[(H1):]$H=H(x, \xi)$ is a $N\times N$ field of Hermitian  matrices in $S(1)$, and there exist  $1 \le m\le N$ and $C>0$  such that for all $ (x, \xi) \in \R^d\times\R^d$, the matrix $H(x, \xi)$ has  $m$ real semi-simple eigenvalues $\lambda_k(x, \xi)$, for $1\leq k \leq m$, with 
$$|\lambda_j(x, \xi) -\lambda_k (x, \xi)|>C \quad \text{for all} \;  1 \le j\neq k  \le m \le N.$$
\item[(H2):] $A=A(x, \xi)$ is a $N\times N$ field of  matrices in $S(1).$
\end{itemize}
 \end{assumption}
 As a consequence of (H1)
the matrix $\mathcal{U}(x, \xi) :=(\textbf{r}_\ell)_{\ell= 1 ,\cdots N}$ of the right normalized eigenvectors  $(\textbf{r}_\ell)_{\ell= 1 ,\cdots N}$ of $H(x, \xi)$ 
is in $S(1)$, takes values which are unitary matrices, and 
\begin{equation}
  \label{commut}  \hU^* H =\D \hU^*=(\lambda_\ell \textbf{r}_\ell)^*_{\ell=1, \cdots N} ,
\end{equation}
%As a consequence of (H1) 
% there exists a symbol $\hU(x,\xi)$, which is unitary everywhere, and a family of constant orthogonal projectors $\left(\Pi_k\right)_{1\le k\le m}$ such that
and
 \begin{equation}\label{def:D U and Pi}
    \D(x,\xi):=(\hU^*H\hU)(x,\xi)=\sum_{k=1}^m \lambda_k(x,\xi)\Pi_k\in S(1),
\end{equation}
where the $\left(\Pi_k\right)_{1\le k\le m}$ are the constant orthogonal projectors onto $\ker(\mathcal D(x,\xi)-\lambda_k(x,\xi)\Id),$ which are  independent of $(x,\xi)\in\R^d$. \medskip

Our aim is to answer to Question \ref{qul} in the case of the system \eqref{hyper-eps}, that is to determine, as $\eps \rightarrow 0$, the limit  observability cost $\mathcal C[\bu_0^\varepsilon]$, as given by Definition \ref{def:mcost}, 
of a family of initial data $(\bu_0^\varepsilon)_{\eps \in (0,1]}$  in $L^2(\R^d ; \R^N)$ which is localized at 
$(x_0 ,\xi_0,b_0) \in \R^d \times (\R^d\setminus\{0\})  \times \R^N$, 
for the corresponding solution $\bu$ of the system \eqref{hyper-eps}
and an observability operator $B\in S(1)$.
The main result of this section provides an answer to this question which generalizes the earlier case where a single characteristic was considered. 
Again the result involves a few ingredients, which we now recall in this more general setting. First it  makes use of the Hamiltonian flows associated with each eigenvalue $\lambda_k$, with $ 1\le k\le m$, as  given by Definition \ref{biflotH}. 
For sake of clarity we  explicitely recall this definition below, which involves nonlinear ODEs.

%%%%%
\begin{Definition}\label{flow-HK}
For $1 \le k \le m$, 
for any  $(x,\xi)$, 
we define 
the Hamiltonian flow $\big( X_{k,x,\xi}, \Xi_{k,x,\xi} \big)$
 induced by the $k$-th eigenvalue $\lambda_k(x, \xi)$ as the unique smooth solution of the following Cauchy problem: 
%, defined by
\begin{equation}
\label{bich1 k}
\big( X_{k,x,\xi}, \Xi_{k,x,\xi} \big)'
= v_{\lambda_k}\big(X_{k,x,\xi}, \Xi_{k,x,\xi}\big),\quad \left(X_{k,x,\xi},\Xi_{k,x,\xi}\right)(0)=(x,\xi).
\end{equation} 
where $v_{\lambda_k}$ is the divergence free vector field from $\R^d\times\R^d$ to itself defined by \begin{equation}
    v_{\lambda_k}=\left(\nabla_\xi\lambda_k,-\nabla_x\lambda_k\right)
\end{equation}
\end{Definition}
%%%%%

We also rely on some Gramian matrices associated with each Hamiltonian flow
and with $H$ and $A$, similarly to Definition \ref{flot-Gr}, but with some surprising extra terms, which encode the influence between the various characteristic fields of the system. More precisely, for $1\le k\le m$, let 
 \begin{align}\label{def:Aeff}
        A^{\rm{eff}}_k &:=\Pi_k \hU^*A\hU \Pi_k+ \frac12 \Pi_k \Big( \{\hU^*,H\}\hU-\{\D,\hU^*\}\hU \Big)\Pi_k,
 \\    \label{def:Beff} B^{\rm{eff}}_k &:=B\hU\Pi_k,
\end{align}
   which both belong to $S(1)$. We may now give the definition of the Gramian matrices, which involves linear ODEs. 
\begin{Definition}\label{flot-Gr-sys}
Let  $x_0 ,\xi_0 \in\R^d$ and  $1\le k\le m$.
We define the $k$-th  Gramian matrix microlocalized in $(x_0 ,\xi_0 )$  as the unique solution $G_{k,x_0 ,\xi_0}$ to the differential  Lyapunov equation 
$$Lyap(A_{k,x_0,\xi_0}  ,B_{k,x_0,\xi_0} ),$$
%
%$$Lyap(A_{k}^{\rm{eff}}(X_{k,x_0,\xi_0}(t),\Xi_{k,x_0,\xi_0}(t)), B_{k}^{\rm{eff}}(X_{k,x_0,\xi_0}(t),\Xi_{k,x_0,\xi_0}(t))),$$
%
where 
$$A_{k,x_0,\xi_0} (t) := A_{k}^{\rm{eff}}(X_{k,x_0,\xi_0}(t),\Xi_{k,x_0,\xi_0}(t))
 \,   \text{ and }   \, 
 B_{k,x_0,\xi_0} (t) := B_{k}^{\rm{eff}}(X_{k,x_0,\xi_0}(t),\Xi_{k,x_0,\xi_0}(t))).$$
\end{Definition}
We also define the bicharacteristic  amplitudes, which are given as solutions to linear ODEs.
\begin{Definition}\label{biflotHk}
For any  $x_0,\xi_0\in\R^d$, for any $\bb_0$ in $\R^N$  and  $1\le k\le m$, 
we define  the  $k$-th  bicharacteristic  amplitude $\bb_{k,x_0,\xi_0,\bb_0}$ 
 as the unique solution to the linear ODE:
  \begin{equation*}
\bb_{k,x_0,\xi_0,\bb_0}'=- A_{k,x_0,\xi_0} \,  \bb_{k,x_0,\xi_0,\bb_0},\quad \bb_{k,x_0 ,\xi_0 ,\bb_0}(0)=\bb_0 .
\end{equation*}
\end{Definition}

The main result of this section is that the inverse of  
observability cost $\mathcal C[\bu_0^\eps]$ for a family of initial data $(\bu_0^\eps)_{\eps\in(0,1]}$ in $L^2(\R^d;\C^N)$ which is localized at $(x_0,\xi_0,\bb_0)$,  in the sense of Definition \ref{def:loc},  converges to a weighted mean of Rayleigh quotients of  the $m$ Gramian matrices with the bicharacteristic amplitudes. 
\begin{thm}\label{thm1 micro}
Let  $(x_0 ,\xi_0,b_0) \in \R^d \times (\R^d\setminus\{0\})  \times \R^N$.
 For $1\le k\le m$, let $G_{T,k}$ be the Gramian matrix given at time $T$, that is  $G_{T,k} := G_{k,x_0,\xi_0} (T),$
  and $\bb_{T,k}$ 
  the characteristic amplitude at time $T$, that is $\bb_{T,k} := \bb_{k,x_0,\xi_0,\bb_0} (T).$    
 Let $(\bu_0^\varepsilon)_{\eps \in (0,1]}$  in $L^2(\R^d ; \R^N)$ be localized at $(x_0 ,\xi_0,b_0)$ in the sense of Definition \ref{def:loc}. 
Then the observability cost $\cC[\bu_0^\eps]$ verifies, as $\eps\to0$, 
\begin{equation}
    \cC[\bu_0^\eps]^{-1}\to \frac{\sum_{k=1}^m \bb_{T,k}^* G_{T,k} \bb_{T,k}}{\sum_{k=1}^m |\bb_{T,k}|^2}.
\end{equation}
%$(X_{k, x, \xi}(t), \Xi_{k, x, \xi}(t))$ is the Hamiltonian flow induced by $\lambda_k$, defined by \eqref{bich1 k} and \eqref{bich2 k}, and $\bb_{T,k}$ is the solution at final time $T$, of 
\end{thm}
Theorem \ref{thm1 micro} extends to the general case of hyperbolic systems, without characteristics crossing,  some earlier results by \cite{cui,DL,ll} in the case of the wave equation  and of the Klein-Gordon equation. 
In these works, the authors use some Egorov type methods, for scalar equations. 
Despite some counterparts of the  Egorov theorem for systems have been established in \cites{assal,BG,cordes}, in the proof below  we prefer 
to extend the approach performed in the previous sections, an approach which we believe is more flexible, direct and general.

\subsection{Proof of Theorem \ref{thm1 micro}}
The first step of the proof is to rewrite the system in normal form. To this end, we need the two preliminary results, that we will prove later.
Recall that $\D$ is defined in \eqref{def:D U and Pi}, and set %$A^\flat\in S(1)$ is given by 
\begin{equation}\label{eq:D1}
        2A^\flat(x,\xi) := \{\hU^*,H\}\hU-\{\D,\hU^*\}\hU ,
    \end{equation}
    which is in $S(1)$. 
\begin{Lemma}\label{lem:diag}
    Let $\bu$ be the solution to \eqref{hyper-eps} with initial datum $\bu_0$, and let $\wt\bu:=\w[\hU^*]\bu$. Then $\wt\bu$ solves the system
    \begin{equation}\label{eq:u tilde}
        \left\{\begin{aligned}
            \de_t\wt\bu+\frac i\eps\w[\D]\wt\bu+\w[\Gamma]\wt\bu&=\eps\w[r_\eps]\wt\bu,\\
            \wt\bu(0)&=\w[\hU^*]\bu_0,
        \end{aligned}\right.
    \end{equation}
    where $ \Gamma=\hU^*A\hU+A^\flat$ and the remainder $r_\eps\in S(1)$ is such that $\mathcal N_{\gamma d}(r_\eps)=O(1)$.
\end{Lemma}
For all $(x,\xi)\in\R^{2d}$, we set
 \begin{equation}\label{defQ}
Q(x,\xi):=-
\sum_{(j,k) / \, j\ne k}  \frac{i\Pi_j\Gamma(x,\xi)\Pi_k}{\lambda_j(x,\xi)-\lambda_k(x,\xi)}.
 \end{equation} 
Then $Q\in S(1)$. Recall that $A^{\rm{eff}}_k$ is given by
    \eqref{def:Aeff}.
\begin{Lemma}\label{lem:block diag}
    Let $\wt\bu$ be the solution to \eqref{eq:u tilde}. %There exists a symbol $Q(x,\xi)\in S(1)$ such that, letting 
    We set $\bv:=(\Id+\eps\w[Q])\wt\bu$ and $\bv_k:=\Pi_k\bv$ for $1\le k\le m$.
    Then
    \begin{equation}\label{eq:vk}
        \left\{\begin{aligned}
            \de_t\bv_k+\frac i\eps\w[\lambda_k\Id]\bv_k+\w[A^{\rm{eff}}_k]\bv_k&=\eps\w[r_\eps]\bv,\\
            \bv_k(0)&=(\Id+\eps\w[Q])\w[\hU^*]\bu_0,
        \end{aligned}\right.
    \end{equation} 
    where   $r_\eps$ is again a remainder in $S(1)$ such that $\mathcal N_{\gamma d}(r_\eps)=O(1).$
\end{Lemma}
Let us take these two results as granted and finish the proof of Theorem \ref{thm1 micro}. 
Rewriting both energy terms involved in the observability cost in terms of $\bv$ gives, using that $\hU\hU^*=\hU^*\hU=\Id$, 
\begin{align}
    \|\bu(T)\|_{L^2}^2&    =\sum_{k=1}^m\|\bv_k(T)\|_{L^2}^2(1+O(\eps)),\label{en T} \\ \|\w[B]\bu\|^2_{L^2}
    &=\sum_{k=1}^m\|\w[B^{\rm{eff}}_k]\bv_k\|_{L^2}^2+\sum_{j\ne k}\left\langle \w[B\hU]\bv_j,\w[B\hU]\bv_k\right\rangle_{L^2}+O(\eps)\|\bv\|_{L^2}^2\label{en obs} .
\end{align}
Note that we split the diagonal and off-diagonal terms in the observed energy because they exhibit very distinct properties.\\

For the off-diagonal terms, we rely on the following lemma.

\begin{Lemma}\label{lem:off-diag}
    Let $\Lambda(x,\xi)\in S(1)$ be a matrix-valued symbol. Let $1\le j,k\le m$ such that $j\ne k$. Then it holds \begin{equation}
        \int_0^T\langle\w[\Pi_k\Lambda\Pi_j]\bv(t),\bv(t)\rangle_{L^2}\,\dd t=O(\eps)\|\bv(T)\|^2_{L^2}.
    \end{equation} 
\end{Lemma}

Applying this lemma with $\Lambda=\hU^*B^*B\hU$ yields, for $1\le j\ne k\le m$, 
\begin{equation}\label{estim:off-diag}
    \left\langle\w[B\hU]\bv_j,\w[B\hU]\bv_k\right\rangle_{L^2}=O(\eps)\|\bv(T)\|_{L^2}^2.
\end{equation}

For the diagonal terms, for each $1\le k\le m$, we  apply Theorem \ref{PDO transport} to the system \eqref{eq:vk}, for $1\le k\le m$, to obtain that 
%If $(\bu_0^\eps)$ is localized in $(x_0,\xi_0,\bb_0)$ then, plugging $\bv_k=\Pi_k(\Id+\eps\w[Q])\w[\hU^*]\bu$, it follows, for $1\le k\le m$, that $(\bv_{k,0}^\eps)$ is localized in $(x_0,\xi_0,\Pi_k\hU^*(x_0,\xi_0)\bb_0))$.
%Theorem \ref{PDO transport}
%, and in particular the estimates \eqref{est fin1} and \eqref{est fin2} then yields 
\begin{align}
    \|\bv(T)\|_{L^2}^2&=|\bb_{T,k}|^2+o(1),\label{estim:diag1}\\
    \int_0^T\|\w[B\hU]\bv_k\|_{L^2}^2&=\bb_{T,k}^*G_{T,k}(T)\bb_{T,k}+o(1).\label{estim:diag2}
\end{align} 

Combining \eqref{estim:off-diag}, \eqref{estim:diag1} and \eqref{estim:diag2}, and substituting them into \eqref{en T} and \eqref{en obs} yields \begin{align*}
    \|\bu(T)\|_{L^2}^2&=\sum_{k=1}^m |\bb_{k,x_0,\xi_0,\bb_0}(T)|^2+o(1),\\
    \int_0^T\|\w[B]\bu\|_{L^2}^2&=\sum_{k=1}^m\bb_{k,x_0,\xi_0,\bb_0}(T)^*G_{k,x_0,\xi_0}(T)\bb_{k,x_0,\xi_0,\bb_0}(T)+o(1).
\end{align*}
Plugging those estimates into the definition of the cost as given by \eqref{costuf}  concludes the proof.

%%%%%
\subsection{Diagonalisation of the propagation term. Proof of Lemma \ref{lem:diag}}
We  consider the system \eqref{hyper-eps}, and apply to it the operator $\w[\hU^*]$, giving
\begin{equation}\label{hyper eps U}
    \de_t\wt\bu+\frac i\eps\w[\hU^*]\w[H]\bu+\w[\hU^*]\w[A]\bu= 0.
    %\eps\w[r_\eps]\bu.
\end{equation}
Applying Lemma \ref{lem:comp} for the identity \eqref{commut}, we obtain \begin{equation}
    \w[\hU^*]\w[H]+\frac\eps{2i} \w[\{\hU^*,H\}]=\w[\D]\w[\hU^*]+\frac\eps{2i}\w[\{\D,\hU^*\}]+\eps^2\w[r_\eps],
\end{equation}
where $r_\eps\in S(1)$ is a remainder symbol such that $\mathcal N_{\gamma d}(r_\eps)$ is bounded uniformly in $\eps$. By a slight abuse of notation, we will denote all such remainders $r_\eps$, despite them not being the same in all the equalities. 
Rearranging, using that $\hU\hU^*=\Id$ and recalling the definition of $A^\flat$ in \eqref{eq:D1}, it follows that 
\begin{equation}\label{comm ord 1}
    \w[\hU^*]\w[H]=\w[\D]\w[\hU^*]+\frac\eps i\w[A^\flat]\w[\hU^*]+\eps^2\w[r_\eps].
\end{equation}
Meanwhile, for the term of order 0, we have \begin{equation}\label{comm ord 0}
    \w[\hU^*]\w[A]=\w[\hU^*A\hU]\w[\hU^*]+\eps\w[r_\eps].
\end{equation}
Substituting  \eqref{comm ord 1} and \eqref{comm ord 0} into \eqref{hyper eps U}, it finally yields \begin{equation}
    \de_t\wt\bu+\left(\frac i\eps\w[\D]+\w[A^\flat]+\w[\hU^*A\hU]\right)\w[\hU^*]\bu=\eps\w[r_\eps]\wt\bu,
\end{equation} 
and the claim immediately follows.

\begin{remark}
    For any $1\le k\le m$, multiplying equality \eqref{comm ord 1} on the left by $\Pi_k$ and on the right by $\w[\hU]\Pi_k$ gives, denoting by $\sigma_p$ the principal symbol of a pseudo-differential operator, \begin{equation}\label{eq:D1 proj}
        \sigma_p\left(\Pi_k(\w[\hU^*]\w[H]\w[\hU] - \w[\D]\w[\hU^*]\w[\hU])\Pi_k\right)=-i\eps \Pi_kA^\flat\Pi_k.
    \end{equation}
    By property \eqref{eq:adjoint} of the Weyl quantization, the first operator in the left-hand side of \eqref{eq:D1 proj} is self-adjoint, and the second one is also self-adjoint, as a consequence of the symbolic identity $(\Pi_k \D)(x, \xi)=( \lambda_k \Pi_k)(x, \xi)$. Therefore the symbolic matrix $i(\Pi_k A^\flat \Pi_k)(x, \xi)$ is Hermitian for $1 \le k \le m$.
\end{remark}

\subsection{Block-diagonalisation of the amplification term. Proof of Lemma \ref{lem:block diag}}

We follow a method initially used by Poincar\'e in the setting of ODEs to further diagonalize the system, into a pseudo-differential normal form where  the term of order $\frac1\eps$ is diagonal and the term of order $1$ is block-diagonal.
Since $(\Id+\eps\w[Q])^{-1}=\Id-\eps\w[Q]+\eps^2\w[r_\eps]$, we get
\begin{align}
    (\Id+\eps\w[Q])\left(\w\left[\frac i\eps\D+\Gamma\right]\right)(\Id+\eps\w[Q])^{-1}&=\w\left[\frac i\eps\D-i[\D,Q]+\Gamma\right]\notag\\&\quad+\eps\w[r_\eps].
\end{align}
Now, we observe from the definition of $Q$ in \eqref{defQ} that
\begin{align}
    -i[\D,Q]+\Gamma=\sum_{k=1}^m\Pi_k\Gamma\Pi_k,\label{prop q}
\end{align}
Substituting \eqref{prop q} into this equality then yields \begin{align}
    (\Id+\eps\w[Q])\left(\w\left[\frac i\eps\D+\Gamma\right]\right)(\Id+\eps\w[Q])^{-1}&=\w\left[\frac i\eps\D+\sum_{k=1}^m\Pi_k\Gamma\Pi_k\right]\notag\\&\quad+\eps\w[r_\eps]
\end{align}
Performing the change of variable $\bv=(\Id+\eps\w[Q])\wt\bu$ in the system \eqref{eq:u tilde} then gives
\begin{equation} \label{eq-v-fi}
    \de_t\bv+\w\left[\frac i\eps\D+\sum_{k=1}^m\Pi_k\Gamma\Pi_k\right]\bv=\eps\w[r_\eps]\bv.
\end{equation}
For $1\le k\le m$, multiplying this equation on the left by $\Pi_k$ then leads to the system \eqref{eq:vk} and the claim is proven.

\subsection{Estimate of the off-diagonal energy terms. Proof of Lemma \ref{lem:off-diag}}

By \eqref{eq-v-fi} and standard pseudo-differential calculations, for any $1\le j\ne k\le m$, recalling that $\lambda_{j}$ denotes the eigenvalues of $\mathcal D$, 
\begin{align*}
\frac{\dd}{\dd t}\left\langle\w\left[\frac{\Pi_k\Lambda\Pi_j}{\lambda_j-\lambda_k}\right]\bv,\bv\right\rangle&=-
\frac i\eps\left\langle\w\left[\left[\frac{\Pi_k\Lambda\Pi_j}{\lambda_j-\lambda_k}, \mathcal D\right]\right]\bv,\bv\right\rangle + O(1)\|\bv\|_{L^2}^2%
%\\&=-\frac i\eps E_{j,k}(t)+O(1)\|\bv\|_{L^2}^2.
\end{align*}
Now, we observe that, for any $1\le j\ne k\le m$,
$$\left[\frac{\Pi_k\Lambda\Pi_j}{\lambda_j-\lambda_k}, \mathcal D\right] = \Pi_k\Lambda\Pi_j ,$$
and we set 
$$E_{j,k}(t):=\langle\w[\Pi_k\Lambda\Pi_j]\bv,\bv\rangle_{L^2}.$$ 
This immediately yields \begin{equation}E_{j,k}(t)=i\eps\frac{\dd}{\dd t}\left\langle\w\left[\frac{\Pi_k\Lambda\Pi_j}{\lambda_j-\lambda_k}\right]\bv,\bv\right\rangle + O(\eps)\|\bv(t)\|^2_{L^2}\end{equation}
Then, integrating in time yields\begin{align*}
    \int_0^TE_{j,k}(t)\,\dd t&=i\eps\left\langle\w\left[\frac{\Pi_k\hU^*B^*B\hU\Pi_j}{\lambda_j-\lambda_k}\right]\bv(T),\bv(T)\right\rangle_{L^2}\\&\quad-i\eps\left\langle\w\left[\frac{\Pi_k\hU^*B^*B\hU\Pi_j}{\lambda_j-\lambda_k}\right]\bv(0),\bv(0)\right\rangle_{L^2}+ O(\eps)\|\bv(T)\|_{L^2}^2 .
\end{align*}
By the non-crossing condition, $\w\left[\frac{\Pi_k\hU^*B^*B\hU\Pi_j}{\lambda_j-\lambda_k}\right]$ is a bounded operator on $L^2$, and the claim immediately follows.

%\subsection{Case of the hyperbolic systems}
% One important class of systems for which the previous analysis can be applied is the class of first-order symmetric hyperbolic systems:
% \begin{align}\label{symmetric}\de_t \bu + \sum_{j=1}^d A_j (x) \de_{x_j}  \bu  + \Gamma(x, D) \bu=0,  \end{align}
% where  the matrix $\A(x, \xi):= \sum_{j=1}^d A_j(x) \xi_j$ is smooth, symmetric and has real semi-simple eigenvalues for all $(x, \xi)$. Of course the analysis applies as well to Friedrich symmetrizable hyperbolic systems and to first-order hyperbolic systems \eqref{symmetric} with semi-simple eigenvalues (strongly hyperbolic systems, \cite{metivier2008}) of constant multiplicity for all $(x, \xi) \in \R^{2d}$, which always admit a pseudodifferential symmetrizer, i.e. a matrix valued strictly positive symbol $A_0=A_0(x, \xi)$ such that$$A_0(x, \xi) \sum_{j =1}^d A_j(x) \xi_j \quad \text{is symmetric},$$ with a few adaptations.

%%%%%%%%%%%%%%%%%%%%%%%%%%%%%%%%%%%%%%%%%%

\section{Compressible gas-dynamics in a rapidly rotating frame}

As an illustration we consider the non-isentropic compressible Euler equations for perfect gases in a rotating frame (see for instance \cites{benzoni,fer}) in $2$D: 
 \begin{equation}\label{inc eul}
    \left\{
    \begin{aligned}
    \de_t\bu+(\bu\cdot\nabla)\bu+\rho^{-1}\nabla p&=0,\\
    \de_tp+\bu\cdot\nabla p+\rho c^2\nabla\cdot\bu+\frac{b(x_2)}{\rm{Ro}}\bu^\perp&=0,\\
    \de_ts+\bu\cdot\nabla s&=0,
    \end{aligned}
    \right.
\end{equation}
where $\rho$ is the density, $\bu$ the velocity, $p$ the pressure, $s$ the entropy, $\rm{Ro}\in\R^+$ the Rossby number, $\bu^\perp=(-u_2,u_1)$, $b$ is a smooth function of $x_2$ and $c$ is the sound speed.
We consider the  the case of polytropic gases, for which the pressure law is 
$p=(\gamma-1)\rho^{\gamma}e^{\frac{s}{c}},$ for $\gamma>1$, and 
  the sound speed $c$ is given by $c=\sqrt{\gamma p/\rho}$. 
  Then, setting
   $\pi:=(p/(\gamma-1))^{2\gamma},$
   the  system \eqref{inc eul} is recast as \begin{equation*}
    \left\{\begin{aligned}
    \de_t\bu+(\bu\cdot\nabla)\bu+2\gamma e^{\frac s{\gamma c}}\pi\nabla\pi+\frac{b(x_2)}{\rm{Ro}}\bu^\perp&=0,\\
    \de_t\pi+\bu\cdot\nabla\pi+\frac{\gamma-1}2\pi\nabla\cdot\bu&=0,\\
    \de_ts+\bu\cdot\nabla s&=0.
    \end{aligned}\right.
\end{equation*}
The linearization of the system at a steady solution $(\overline{\bu}, \overline{\pi}, \overline{s})$, where $\overline{\bu}=(\overline{u}_1(x_2), 0)$ is a shear flow, $\overline{s} \in \R^+$ is a constant value and $\overline{\pi}=\overline{\pi}(x_2)$ satisfies the relation
$2\gamma \text{Ro} \exp(\overline{s}/c\gamma) (\overline{\pi} \overline{\pi}')(x_2)= -b(x_2) \overline{u}_1(x_2)$. 
In the case where $\text{Ro}=\eps$, the semi-classical approximation  
yields the following  system: 
%\begin{equation}\label{euler coriolis} \left\{\begin{aligned}\de_t \bu+(\overline u_1 \de_1 )\bu+(\bu\cdot\nabla)\overline \bu + 2\gamma\exp({\overline s}/{c\gamma})\left(\overline\pi\nabla\pi+\pi\nabla\overline\pi+c^{-1}\gamma^{-1} {s}\overline\pi\nabla\overline\pi\right)+\frac{b(x_2)}{\eps} \bu^\bot&=0,\\ \de_t\pi+\overline u_1 \de_1 \pi+u_2 \overline\pi'+\frac{\gamma-1}2\overline\pi\Div(\bu)&=0,\\ \de_t s+\overline u_1\de_1  s&=0. \end{aligned} \right.\end{equation}
%In compact formulation, one has 
\begin{equation}\label{eul cor weyl}
\de_t \textbf{U}+\w\left[\frac i\eps H +  A\right]\textbf{U}= 0,
\end{equation}
where
 $$\textbf{U}=\begin{pmatrix}u_1\\u_2\\\pi\\s\end{pmatrix}, \quad H=\begin{pmatrix}\overline u_1(x_2) \xi_1 &-ib(x_2) &2\gamma e^{\frac{\overline s}{\gamma c}}\overline\pi \xi_1&0\\ ib &\overline u_1(x_2) \xi_1 & 2\gamma e^{\frac{\overline s}{\gamma}}\overline\pi \xi_2&0\\ \frac{\gamma-1}2\overline\pi\xi_1&\frac{\gamma-1}2\overline\pi\xi_2&\overline u_1(x_2) \xi_1 &0\\0&0&0&\overline u_1(x_2) \xi_1 \end{pmatrix},$$
and
%M
$$A =\begin{pmatrix}0 &\overline u_1'(x_2) &0&0\\0&0&2\gamma e^{\frac{\overline s}{\gamma}} \overline\pi'(x_2) &2\gamma e^{\frac{\overline s}{\gamma}}\overline\pi\overline\pi'(x_2)\\0&\overline\pi'(x_2)&0&0\\0&0&0&0 \end{pmatrix} .$$

The principal symbol $H$ is not symmetric; however, setting
 $$S :=\diag(1, 1, 4\frac{\gamma}{\gamma-1}\exp({\overline s}/{c\gamma}),1),$$
 which only depends on $x$, 
 and computing 
 $$\wt H=\sqrt{S}H\sqrt{S}^{-1}=\begin{pmatrix}\overline u_1\xi_1 &-ib&\eta \xi_1&0\\ ib &\overline u_1\xi_1 & \eta \xi_2&0\\ \eta\xi_1&\eta\xi_2&\overline u_1\xi_1 &0\\0&0&0&\overline u_1\xi_1 \end{pmatrix},$$
 where  %$C_1=\sqrt{\gamma(\gamma-1)}$ 
 $\eta := \sqrt{\gamma(\gamma-1)} \exp(\frac{\overline s}{2c\gamma})\overline\pi$ 
 and $\sqrt{S}$ is the positive square root of $S$, we observe that $ \wt H$
 is symmetric. 
 We are then led to a system of the form: 
\begin{equation}\label{eul cor sym}
\de_t\textbf{W}+\w\left[\frac i\eps\wt H+  \wt A\right]\textbf{W}=\eps\w[r_\eps]\textbf{W} .
\end{equation}
%
%for some computable $\wt A_0$. %can be obtained by formula \eqref{gamma tilde}. 
The  eigenvalues of  $\wt H$  are 
%semi-simple with constant multiplicity: 
%
$\lambda_1=\overline u_1(x_2)\xi_1$ 
with multiplicity  $2$, $$ \lambda_{2}= \overline u_1\xi_2 + \sqrt{b^2+C_1^2\overline\pi^2 \exp({\overline s}/{c\gamma})|\xi|^2},$$
with  multiplicity $1$, 
and 
$$ \lambda_{3} = \overline u_1\xi_2 - \sqrt{b^2+C_1^2\overline\pi^2 \exp({\overline s}/{c\gamma})|\xi|^2}  ,$$
with  multiplicity $1$.
The non-crossing assumption is verified for $(x_1, x_2)$ such that 
$$b^2(x_2)+C_1^2\overline\pi^2(x_2)\exp({\overline s}/{c\gamma})|\xi|^2\ne0.$$

%Since $\lambda_{2}, \lambda_3 $ are simple eigenvalues, the related blocks $G_{2,x,\xi}(t), G_{3,x,\xi}(t)$ of the Gramian are fully diagonal matrices so that they will not contribute to the observability cost as the only add a correction to the oscillating phase. Therefore we focus on the double eigenvalue $\lambda_1$ to show the following.

Two orthogonal eigenvectors of $\lambda_1$ are given by 
$$(0,0,0,1)^T,\quad  \frac1{\langle\xi\rangle_b}(-i\xi_2, i\xi_1, \wt b, 0)^T,\quad \text{where}  \quad \langle\xi\rangle_b=\sqrt{|\xi^2|+\wt b^2}, \quad \wt b=\frac{b}{C_1\overline\pi}\exp(-{\overline s}/{2c\gamma}).$$
We can then compute the block  of the (block-diagonal) matrix $A^\flat(x, \xi)$ in \eqref{eq:D1} associated to $\lambda_1$
\begin{align*}
 A^{\flat}_{\lambda_1}(x,\xi)=\begin{pmatrix}
0 & 0 \\
0 & - \frac{i\xi_1  C_1 \exp(\overline{s}/2\gamma) \overline{\pi}\de_2 \wt{b}(x_2)}{\langle \xi \rangle^2_b}
\end{pmatrix}.
\end{align*}

%%%%%%%%%%%%%%%%%%%%%%%%%%%%%%%%%%%%%%%%%%%%%%%%%
%%%%%%%%%%%%%%%%%%%%%%%%%%%%%%%%%%%%%%%%%%%%%%%%%
%%%%%%%%%%%%%%%%%%%%%%%%%%%%%%%%%%%%%%%%%%%%%%%%%
%%%%%%%%%%%%%%%%%%%%%%%%%%%%%%%%%%%%%%%%%%%%%%%%%

\bigskip \ \par \noindent {\bf Acknowledgements.} F. S. was partially supported by the Agence Nationale de la Recherche, Project SINGFLOWS, grant ANR-18-CE40-0027-01,  Project BOURGEONS, grant ANR-23-CE40-0014-01, and Institut Universitaire de France. 
R.B. acknowledges financial support by the Italian Ministry for Universities and Research (MUR), Project PRIN 2022 TESEO, grant 2022HSSYPN.  

%%%%%

\end{document}